\documentclass[12pt]{amsart}
\usepackage{amscd,amssymb,amsthm,amsmath,amssymb,textcomp,supertabular,rotating,longtable,enumerate,rotating,mathrsfs,mathtools,multirow, color, tikz}
\usepackage[matrix,arrow,curve]{xy}
\usepackage{hyperref}

\newcommand{\PP}{{\mathbb P}}
\newcommand{\CC}{{\mathbb C}}

\newtheorem{theorem}[equation]{Theorem}
\newtheorem*{theorem*}{Theorem}

\newtheorem{proposition}[equation]{Proposition}
\newtheorem*{proposition*}{Proposition}
\newtheorem{lemma}[equation]{Lemma}
\newtheorem{corollary}[equation]{Corollary}
\newtheorem*{corollary*}{Corollary}

\newtheorem*{problem*}{Problem}

\newtheorem*{question*}{Question}

\newtheorem*{construction*}{Construction}
\newtheorem*{maintheorem*}{Main Theorem}

\newcounter{iequation}

\newtheorem{itheorem}[iequation]{Theorem}
\newtheorem{iquestion}[iequation]{Question}
\newtheorem{iexample}[iequation]{Example}

\theoremstyle{remark}
\newtheorem{iremark}[iequation]{Remark}

\theoremstyle{definition}
\newtheorem{example}[equation]{Example}
\newtheorem*{example*}{Example}
\newtheorem{definition}[equation]{Definition}
\newtheorem*{definition*}{Definition}

\theoremstyle{remark}
\newtheorem{remark}[equation]{Remark}
\newtheorem*{remark*}{Remark}

\makeatletter\@addtoreset{equation}{section} \makeatother

\makeatletter\@addtoreset{equation}{section} \makeatother

\makeatletter\@addtoreset{iequation}{section} \makeatother

\author{Sukjoo Lee,
	Victor Przyjalkowski}

\title{Fibers of Landau--Ginzburg models and rationality}

\address{\emph{Sukjoo Lee}
	\newline
	\textnormal{Center for Geometry and Physics, Institute for Basic Science (IBS), Pohang 37673, South Korea.}
	\newline
	\textnormal{\texttt{sukjoo216@ibs.re.kr}}}

\address{\emph{Victor Przyjalkowski}
	\newline
	\textnormal{Steklov Mathematical Institute of Russian Academy of Sciences, 8 Gubkina street, Moscow 119991, Russia.}
	\newline
	\textnormal{National Research University Higher School of Economics, Russian Federation, Laboratory of Mirror Symmetry, NRU HSE, 6 Usacheva str., Moscow, Russia, 119048.}
	\newline
	\textnormal{\texttt{victorprz@mi-ras.ru, victorprz@gmail.com}}}

\begin{document}
	
	\begin{abstract}
		In this article, we study how the rationality of a Fano threefold is reflected in its standard mirror Landau--Ginzburg model and its deformations. The main result is that a Fano threefold is rational if and only if the monodromy around every reducible fiber of its generic mirror Landau--Ginzburg model is unipotent.
	\end{abstract}
	
	\maketitle

	\section{Introduction}
Mirror Symmetry provides a modern and challenging perspective on classical algebraic geometry. A mirror of a Fano variety is a so-called Landau--Ginzburg model, whose properties are closely related to those of the original Fano variety. Although most Mirror Symmetry conjectures, such as the Homological Mirror Symmetry conjecture, are very hard to prove, one can still study the geometry of algebraic varieties through their mirror counterparts. In particular, instead of working with the most general formulation, one may adopt a precise definition of these mirror objects, derive concrete numerical conjectures from the general Mirror Symmetry framework, and verify them for specific Fano varieties and their mirrors.

As a definition of the mirror dual objects, we take toric Landau--Ginzburg models or their Calabi--Yau compactifications (for more precise definitions and references see Subsection~\ref{s:CYcptn}). There is a series of conjectures relating their invariants
with geometric invariants of Fano varieties (see~\cite{Prz13},~\cite{PS15},~\cite{KKP17},~\cite{LP16},~\cite{ChP18},~\cite{HKP20},~\cite{ChP20}).
Let us give an example of statements of this kind.

\begin{itheorem}[\cite{Prz13},~\cite{ChP18},~\cite{PS15}]
\label{theorem:Hodge number}
Let $X$ be a smooth Fano variety of dimension $n$. Let $w\colon Y\to \CC$ be a Calabi--Yau compactification of a standard toric Landau--Ginzburg model for a Fano threefold $X$
or a Landau--Ginzburg model of Givental's type for a Fano complete intersection $X$. Then
		\[
		h^{1,n-1}(X)=\sum_{t \in \CC}(\rho_{\lambda}-1),\]
		where $\rho_{\lambda}$ is the number of irreducible components of $w^{-1}(\lambda)$.
\end{itheorem}

Conjecturally the assertion of Theorem~\ref{theorem:Hodge number} is expected to hold for any smooth Fano varieties. Moreover, it has a deep
generalization, called Katzarkov--Kontsevich--Pantev conjectures, see~\cite{KKP17} and~\cite{LP16}.
Roughly, one can define analogues of Hodge numbers $h^{p,q}(Y,w)$ and $f^{p,q}(Y,w)$ for a Landau--Ginzburg model $(Y,w)$,
and if $(Y,w)$ is mirror dual to a Fano variety $X$ with $\dim X=\dim Y=n$, then for all possible $p,q \in \mathbb{Z}$, one should have
$$
h^{p,q}(Y,w)=f^{p,q}(Y,w)=h^{n-p,q}(X).
$$
For the threefold case, this was proven in~\cite{Ha21} and~\cite{ChP18}. For all smooth toric Fano complete intersections, the result was proven in~\cite{HL25}.

Among other conjectures and results of a similar kind, one stays apart. The reason is that conjectures like those presented above can be justified by deriving them from Homological Mirror Symmetry conjecture. However, we do not know how the following result is related to it, since algebraic–geometric properties such as rationality do not fit very well into the mirror picture.

\begin{itheorem}[\cite{KP09}]
\label{theorem:rank 1}
Let $(Y,w)$ be a Calabi--Yau compactification of Landau--Ginzburg model of Picard rank one Fano threefold $X$; let us assume that $X$
is general in its family. Then $w$ has at most one
reducible fiber, and the monodromy around this fiber is unipotent if and only if %a general member of the family where $X$ lies
$X$ is rational.
\end{itheorem}

Note that if we require Dolgachev--Nikulin duality to hold for anticanonical sections of $X$ and fibers of $w$, then $Y$ is unique
(up to flops), see~\cite[Section 4.2]{ILP13} (another approach to uniqueness is given by rigid maximally-mutable Laurent polynomials).
Note also that non-unipotency appears when the reducible fiber is either non-reduced or not simple normal crossings. In this paper, we generalize Theorem~\ref{theorem:rank 1} to the case of higher Picard rank.

Let us recall that there are $105$ deformation families of smooth Fano threefolds. We follow the numeration from~\cite{IP},
	adding the variety found in~\cite{MM04} to the end of the list of Picard rank $4$ threefolds.
	We denote by $X_{m-n}$ a general Fano threefold of Picard rank $m$ and number $n$ from the lists.
	Among $105$ deformation families of three-dimensional Fano manifolds, $98$ have members with very ample anticanonical class.
	According to~\cite[Theorem 4.1]{CKPT21}, mutation-equivalence classes of \emph{rigid maximally-mutable Laurent polynomials}
	(see~\cite{CKPT21}) such that there is a polynomial in each class whose Newton polytope is reflexive, correspond one-to-one to the~$98$ families.
	Let us call them \emph{standard}.
	For each of the rest $7$ families there exist classes of rigid maximally-mutable Laurent polynomials as well, however their uniqueness
	is not proven.
	Thus for these $7$ cases let us choose those of them that are discussed in~\cite{ChP18} and call them standard as well. In \cite{CCGK16}, it is proven that the quantum period of $X$ is equal to the period of corresponding standard Laurent polynomial.

    Hironaka-type arguments immediately show that Calabi–Yau compactifications of any Laurent polynomial within the same mutation class differ only by flops (see Remark~\ref{remark:mutation equivalence}). Therefore, to study the components of fibers, it suffices to choose any element of a standard class; its Calabi–Yau compactification will be referred to as the \emph{standard Landau--Ginzburg model}.

One can ask the analogue of Theorem \ref{theorem:rank 1}; $X$ is rational if and only if for its standard Landau--Ginzburg model $(Y, w)$, the monodromy around any reducible or non-reduced fiber, if such a fiber exists, is unipotent. However, this turns out not to be the case; see Example~\ref{example:not unipotent 2-12}. To address this issue, one needs to consider deformations of a standard Landau--Ginzburg model.

	In~\cite{DHKOP24}, it is proven that the deformation space of the standard Landau--Ginzburg models of a smooth Fano threefold $X$ is unobstructed, and that any small deformation can be obtained by varying the coefficients of the corresponding Minkowski polynomial. Therefore, one can construct a parameter space of Landau--Ginzburg models with $(\mathrm{rk}(\mathrm{Pic}(X))-1)$ parameters. In other words, there is a $\left(\mathrm{rk}(\mathrm{Pic}(X))-1\right)$-dimensional family of Landau--Ginzburg models that are mirror to the corresponding Fano threefolds equipped with general complexified divisor classes (see Definition~\ref{def:parametrized LG}). We refer to such families as \emph{parameterized Landau--Ginzburg models}, and their members as \emph{deformed Landau--Ginzburg models}. A deformed Landau--Ginzburg model is called \emph{generic} if it is a generic element of the family; the original standard Landau--Ginzburg model is called \emph{anticanonical}.
	
	Thus, unlike the Picard rank one case, the unipotency of the monodromy of a Landau--Ginzburg model may change under variation, especially when some singular fibers collide.
\begin{iexample}[{see Example~\ref{eg:2-12} for details}]
\label{example:not unipotent 2-12}
Consider a Fano threefold $X_{2-12}$, a blow-up of $\PP^3$ in genus $3$ curve. One can see that, for a generic deformed Landau–Ginzburg model, there are five singular fibers: one simple normal crossings reducible fiber and four fibers each containing a single ordinary double point. In particular, the monodromy around the reducible fiber is unipotent. However, as the parameter varies toward the anticanonical Landau–Ginzburg model, one of the ordinary double points collides with the reducible fiber, making the monodromy around it non-unipotent. See Example~\ref{eg:2-12} for a proof of this fact.
\end{iexample}

This example justifies the following generalization of Theorem~\ref{theorem:rank 1}, which is the main result of the paper.

	\begin{itheorem}
		\label{theorem:any rank rationality}
		Let 
$X$ be a smooth Fano threefold that is general in the corresponding family. Let $w\colon Y\to \mathbb{C}$ be a generic Landau--Ginzburg model for $X$.
		Then $X$ is rational if and only if either every fiber of $w$ is irreducible and reduced, or the monodromy in the second cohomology at all reducible fibers of $w$ is unipotent.
	\end{itheorem}
	Note that one can get rationality data of smooth Fano threefolds from~\cite[\S\S 12.2-12.6]{IP} and references therein.
	
	\begin{iremark}
		We offer a heuristic interpretation of this theorem from the viewpoint of Mirror Symmetry. Let's regard $X$ on the $B$-side and its anticanonical Landau--Ginzburg model $(Y,w)$ on the $A$-side. Homological Mirror Symmetry predicts an equivalence of the derived category of coherent sheaves of $X$, that is $D^bCoh(X)$ and Fukaya category of $(Y,w)$, that is $Fuk(Y,w)$. The latter category admits a semiorthogonal decomposition indexed by the singular fibers of $w$, which should have a mirror counterpart. A natural expectation is that the component corresponding to the reducible fiber of $w$ mirrors to the Kuznetsov component of $D^bCoh(X)  $. In this picture, the unipotency of the monodromy could be viewed as indicating that the Kuznetsov component is not fractional Calabi--Yau category, while the interpretation to the rationality of $X$ is mysterious (cf. \cite{Kuz09}).
		
		Deformation of Landau--Ginzburg model $(Y,w)$ can be understood as a deformation of symplectic structure on $X$. Such deformations may affect the stability conditions on  $D^bCoh(X)$ and $Fuk(Y,w)$, potentially inducing wall crossings that alter their semiorthogonal decompositions. Theorem~\ref{theorem:any rank rationality} suggests that, for a generic choice of stability condition, the expected mirror correspondence above can be recovered.
\end{iremark}

It is natural to generalize Theorem~\ref{theorem:any rank rationality} to higher dimension cases.
Although there is no classification of Fano varieties in dimensions greater then $3$, complete
intersections often serve as standard test examples.

\begin{iexample}
\label{ex: degree N in P^n}
    Let $X$ be a smooth hypersurface of degree $n$ in $\mathbb{P}^{n}$.  Let $w$ be a Calabi--Yau compactification
of toric Landau--Ginzburg model of Givental's type
    $$
    \frac{(x_1+\cdots+x_{n-1}+1)^n}{\prod_{i=1}^{n-1} x_i}\in \CC[x_1^{\pm 1},\cdots, x_{n-1}^{\pm 1}]
    $$
for $X$, see~\cite{PS15}.
It is clear that the component of the fiber  $w^{-1}(0)$ lying on the torus is of multiplicity $n$, so by Griffiths--Landman--Grothendieck
theorem the monodromy around it is not unipotent. Note that $X$ is not rational by~\cite{Pu}.
\end{iexample}

However the life seems not quite so perfect, due to the following.

\begin{iexample}
\label{ex: cubic fourfold}
    Let $X$ be a smooth cubic fourfold. Let $w$ be a Calabi--Yau compactification
of toric Landau--Ginzburg model of Givental's type
    $$
    \frac{(x+y+1)^4}{xyzt}+z+t\in \CC[x^{\pm 1},y^{\pm 1},z^{\pm 1}, t^{\pm 1}],
    $$
    for $X$, see~\cite{PS15}.
    All its fibers are reduced, and a unique reducible fiber of $w$ consists of two components, see~\cite{KP09}. They intersect along a $K3$ surface, elliptically fibered with three $\widetilde{E}_6$ singular fibers.
    However one can check that %Thus, by Theorem~\ref{theorem: GLG}
    the monodromy around this unique reducible fiber is unipotent.
    Classical conjecture states that (very) general cubic fourfold is not rational (cf.~\cite{KKPY25}), while there is a series of examples of rational smooth cubic fourfolds.
\end{iexample}

Example~\ref{ex: cubic fourfold} shows that the analogue of Theorem~\ref{theorem:any rank rationality} in higher-dimensional cases is likely to be more complicated. Non-unipotency of the monodromy for a generic Landau–Ginzburg model should guarantee non-rationality, whereas in the unipotent case, more subtle considerations need to be taken into account.

\begin{iquestion}\label{Question:higher dimensional}
Is it true that a smooth Fano variety %$X$ of dimension $n$
is not rational if
monodromy of its generic Landau--Ginzburg model %of dimension $n$ for $X$
is not unipotent?
\end{iquestion}

\medskip
The paper is organized as follows. In Section~\ref{section:prelim}, we introduce the preliminaries needed to formulate and prove Theorem~\ref{theorem:any rank rationality}, including toric Landau--Ginzburg models (with parameterized versions) and the construction of their Calabi--Yau compactifications as suggested in~\cite{ChP18}. In Section~\ref{section:proof}, we prove Theorem~\ref{theorem:any rank rationality} by studying Fano threefolds on a case-by-case basis. Subsection~\ref{subsection:rational} focuses on rational Fano threefolds, while Subsection~\ref{subsection:non-rational} deals with non-rational ones. Finally, for the reader's convenience, all data and references are collected in a table in Section~\ref{section:table}.

	\section*{Acknowledgement}
	We would like to thank Andrew Harder for helpful discussions. S.\,L. was supported by the Institute for Basic Science (IBS-R003-D1).
	
	\section{Mirror Landau--Ginzburg models for Fano 3-folds}
	\label{section:prelim}
	
	\subsection{Toric Landau--Ginzburg models and Calabi--Yau compactifications}\label{s:CYcptn}
	
	We first recall the notion of toric Landau--Ginzburg model introduced in ~\cite{Prz13}.
More precise definitions and details can be found in~\cite{Prz18}.
Let $\varphi[f]$ be the constant term of a Laurent polynomial $f$. Define \emph{the main period} for $f$ to be the following formal series: $I_f(t)=\sum \varphi[f^j]t^j$.
	\begin{definition}
		A \emph{toric Landau--Ginzburg model} for a pair of a smooth Fano variety of $X$ of dimension $n$ and a divisor $D$ on it is a Laurent polynomial $f \in \mathbb{C}[x_1^{\pm 1}, \cdots, x_n^{\pm 1}]$ which satisfies the following conditions:
		\begin{enumerate}
			\item (Period condition): $I_{f}=\widetilde{I}_0^{(X,D)}$, where $\widetilde{I}_0^{(X,D)}$ is the so called \emph{restriction of the constant term of regularized $I$-series to the anticanonical direction}.
			\item (Calabi--Yau condition): There exists a relative compactification of a family
			\[
			f:(\mathbb{C}^*)^n \to \mathbb{C}\]
			whose total space is a (non-compact) smooth Calabi--Yau variety $Y$. Such compactification is called a \emph{Calabi--Yau compactification}.
			\item (Toric condition) There is a degeneration $X \rightsquigarrow T_X$ to a toric variety $T_X$ whose fan polytope is same as the Newton polytope of $f$.
		\end{enumerate}
	\end{definition}

Since deformed (and in particular anticanonical) Landau--Ginzburg models are Minkowski ones, the following holds.
	
\begin{theorem}[{see~\cite[Theorem 1]{Prz16}}]
\label{theorem:standard are toric}
	Let $X$ be a smooth Fano threefold, and let 
 $f$ be a deformed Landau--Ginzburg model of $X$. Then $f$ is a toric Landau--Ginzburg model. More precise, the following diagram exists. 	
\begin{equation}\tag{$\oplus$}
		\label{equation:CCGK-compactification}
		\xymatrix{
			(\mathbb C^*)^3\ar@{^{(}->}[rr]\ar@{->}[d]_{f}&&Y\ar@{->}[d]^{w}\ar@{^{(}->}[rr]&&Z\ar@{->}[d]^{u}\\
			\mathbb{C}\ar@{=}[rr]&&\mathbb{C}\ar@{^{(}->}[rr]&&\mathbb{P}^1}
	\end{equation}
	where %we denote the surjective morphism given by $p$ by $w$,
	the variety $Y$ is a smooth threefold with $K_Y\sim 0$,
	and $Z$ is a smooth compact threefold such that
	$$
	-K_Z\sim u^{-1}\big(\infty\big).
	$$
\end{theorem}
\noindent For simplicity, we call $(Y, w)$ a (compactified) Landau--Ginzburg model.

\begin{remark}[{cf.~\cite[after (1.3)]{ChP20} and~\cite[after Problem 13]{KasPr22}}]
\label{remark:mutation equivalence}
Let $f$ and $g$ be two mutationally equivalent toric Landau--Ginzburg models (in particular, one can take $f=g$).
Let $(Y_f,w_f)$ and $(Y_g,w_g)$ be their (log) Calabi--Yau compactifications.
Since  $Y_f$ and $Y_g$ are $K$-equivalent, by~\cite[Theorem 1]{Ka08} they are birational in codimension $2$.
In particular, the number of components and multiplicities
of components of their fibers $w_f^{-1}(p)$ and $w_g^{-1}(p)$ are the same for any $p$.
\end{remark}

	The most natural way to construct a Calabi--Yau compactification is described in~\cite{Prz16}.
	However, to study reducible fibers of $w$, it is more convenient to use compactification method introduced in~\cite{ChP18}. Let us briefly recall it below; more details can be seen in~\cite[Section 1.14]{ChP18}.

	Every deformed toric Landau--Ginzburg model $f$ we are interested in can be homogenized to the quartic $f_4(x,y,z,t)$ in $\mathbb{P}^3$ so that it gives a pencil $\mathcal{S}$ of quartic surfaces on $\mathbb{P}^3$ given by
	\begin{equation*}\label{equation:quartic surface}
		S_\lambda=\{f(x,y,z,t)=\lambda xyzt\}.
	\end{equation*}
	for $\lambda \in \mathbb{C} \cup \{\infty\}$. This expands the commutative diagram \eqref{equation:CCGK-compactification} to the following commutative diagram:
	\begin{equation*}
		\label{equation:diagram}
		\xymatrix{
			(\mathbb{C}^*)^3\ar@{^{(}->}[rr]\ar@{->}[d]_{f}&&Y\ar@{->}[d]^{w}\ar@{^{(}->}[rr]&&Z\ar@{->}[d]^{u}&& V\ar@{-->}[ll]_{\chi}\ar@{->}[d]^{g}\ar@{->}[rr]^{\pi}&&\mathbb{P}^3\ar@{-->}[lld]^{\phi}\\
			\mathbb{C}\ar@{=}[rr]&&\mathbb{C}\ar@{^{(}->}[rr]&&\mathbb{P}^1\ar@{=}[rr]&&\mathbb{P}^1&&}
	\end{equation*}
	where $\phi$ is a rational map given by the pencil $\mathcal{S}$, and $\pi:V \to \mathbb{P}^3$ is an explicit birational map with the threefold $V$ being smooth, and $\chi$ is a composition of flops (which is usually trivial).
	
	Let $\widetilde{S}_\lambda$ be its proper transform of $S_\lambda$ on the threefold $V$, and let $E_1, \cdots, E_n$ be the $\pi$-exceptional divisors for some $n$. Then
	\[
	K_V+\widetilde{S}_\lambda+\sum_{i=1}^n\mathbf{a}^\lambda_i E_i \sim \pi^*(K_{\mathbb{P}^3}+S_\lambda) \sim 0\]
	for some non-negative integers $\mathbf{a}_1^\lambda, \cdots, \mathbf{a}_n^\lambda$. Since $-K_V \sim g^{-1}(\lambda)$, we have
	\[
	g^{-1}(\lambda)=\widetilde{S}_\lambda+\sum_{i=1}^n\mathbf{a}^\lambda_i E_i . \]
	Since $\chi$ is a composition of flops, this equation implies that
	\begin{equation}\label{equation:counting components1} \tag{$\otimes$}
		[u^{-1}(\lambda)]=[S_\lambda]+[{\text{the number of indices } i\in \{1, \cdots, n\} \text{ such that } \mathbf{a}_i^\lambda >0}]
	\end{equation}
	where $[-]$ indicates the number of irreducible components.
	
	The computation of $[S_\lambda]$ is straightforward (e.g.  ~\cite[Lemma 1.5.3]{ChP18}). There are two different types of the other term. On the base locus of the pencil $\mathcal{S}$, there are  irreducible curves $C_1, \cdots, C_r$ and a finite set of points $\Sigma$ which are singular points for $S_\lambda$ for a generic $\lambda \in \mathbb{C}$.
For every base curve $C_j$ and every point $P \in \Sigma$, we define
	\begin{align*}
		&\mathbf{C}_j^\lambda=[{\text{the number of indices } i\in \{1, \cdots, n\} \text{ such that } \mathbf{a}_i^\lambda >0 \text{ and }\pi(E_i)=C_j}], \\
		&\mathbf{D}_P^\lambda=[{\text{the number of indices } i\in \{1, \cdots, n\} \text{ such that } \mathbf{a}_i^\lambda >0 \text{ and }\pi(E_i)=P}].
	\end{align*}
Then the equation \eqref{equation:counting components1} becomes
	\[
	[u^{-1}(\lambda)]=[S_\lambda]+\sum_{j=1}^r\mathbf{C}_j^\lambda+\sum_{P\in \Sigma}\mathbf{D}_P^\lambda.
	\]
We first describe how to compute the number $\mathbf{C}_j^\lambda$ and the multiplicity $\mathbf{a}_i^\lambda$ of $E_j$ that contracts to $C_j$.
	
	For every $\lambda \in \mathbb{C}\cup \{\infty\}$ and every $j \in \{1, \cdots, r\}$, we let
	\[
	\mathbf{M}_j^\lambda=\mathrm{mult}_{C_j}(S_\lambda).\]
	For any two distinct quartic surfaces $S_{\lambda_1}$ and $S_{\lambda_2}$ in the pencil $\mathcal{S}$, we have
	\[
	S_{\lambda_1} \cdot S_{\lambda_2}=\sum_{j=1}^r\mathbf{m}_jC_j\]
	for some positive numbers $\mathbf{m}_1, \cdots, \mathbf{m}_r$. Note that $\mathbf{m}_i \geq \mathbf{M}_i^\lambda$ for every $\lambda \in \mathbb{C}\cup\{\infty\}$.
	
	Since computation of the number $\mathbf{C}_j^\lambda$ and $\mathbf{a}_i^\lambda$ can be checked in a general point of the curve $C_j$, we may assume that $C_j$ is smooth. Since general surfaces in the pencil $\mathcal{S}$ are smooth at general point of the curve $C_j$, there is a $\mathbf{m}_j$-sequence of blow ups of smooth curves
	\[
	V_{\mathbf{m}_j} \xrightarrow{\gamma_{\mathbf{m}_j}} V_{\mathbf{m}_j-1} \xrightarrow{\gamma_{\mathbf{m}_j}-1} \cdots \xrightarrow{\gamma_2} V_1 \xrightarrow{\gamma_1} \mathbb{P}^3\]
	such that $\gamma_k$ is the blow up of the curve $C^{k-1}_j$ where $C^0_j=C_j$ and $C_j^{k-1}$ is a base curve contained in the proper transform of general surfaces in $\mathcal{S}$ on the threefold $V_{k-1}$  with $\gamma_k(C_j^{k})=C_j^{k-1}$.
	
	For every $k \in \{1, \cdots, \mathbf{m}_j\}$, denote by $F_k$ the exceptional surface of the $k$-th blow up $\gamma_k$ and  $S_\lambda^k$ the proper transform of the surface $S_\lambda$ on the threefold $V_k$. Then
	\[\sum_{k=0}^{\mathbf{m}_j-1} \mathrm{mult}_{C_j^k}(S_\lambda^k)=\mathbf{m}_j.\]
	Moreover, for every $i \in \{1, \cdots, n\}$ such that $\beta(E_i)=C_j$, there exists $k=k(i) \in \{1, \cdots, \mathbf{m}_j-1\}$ such that $E_i$ is the proper transform of the divisor $F_k$ on the threefold $V$. Conversely, for every $k \in \{1, \cdots, \mathbf{m}_j-1\}$, there is $i \in \{1, \cdots, n\}$ such that $E_i$ is the transform of the divisor $F_k$. Therefore,
	\[\mathbf{a}_i^\lambda=\sum_{k=0}^{k(i)-1}\left(\mathrm{mult}_{C_j^k}(S^k_\lambda)-1\right).\]
	On the other hand, we also have
	\[
	\mathbf{M}^\lambda_j=\mathrm{mult}_{C_j}(S_\lambda) \geq \mathrm{mult}_{C^1_j}(S_\lambda^1) \geq \cdots \geq \mathrm{mult}_{C^{\mathbf{m}_j-1}_j}(S_\lambda^{\mathbf{m}_j-1})  \geq 0.\]
	This proves the following.
	\begin{lemma}[{\cite[Lemma 1.8.5]{ChP18}}]
\label{lem:multiplicity}
Fix $\lambda\in\mathbb{C}\cup\{\infty\}$ and $j\in\{1,\ldots,r\}$. Then
$$
\mathbf{C}_j^\lambda=\left\{\aligned
&0\ \text{if}\ \mathbf{M}_j^\lambda=1,\\
&\mathbf{m}_j-1\ \text{if}\ \mathbf{M}_j^\lambda\geqslant 2.\\
\endaligned
\right.
$$
	\end{lemma}

The proof of~\cite[Lemma 1.8.5]{ChP18} implies the following.

\begin{corollary}
If $\mathbf{M}_j^\lambda \geq 3$, then $S_\lambda$ has a multiple component.
\end{corollary}

In particular, when $\mathbf{M}_j^\lambda=2$, all multiplicities are $1$ except for the last horizontal divisor, whose multiplicity is $0$. This fact will be used in Lemmas \ref{lemma:1-13}, \ref{lemma:2-5}, \ref{lemma:2-11}.
	
	To compute the number $\mathbf{D}^\lambda_P$, one can partially resolve $P$, which yields new base curves in the base, and apply the same strategy explained above. We omit explicit description and refer to~\cite[Section 1.9, 1.10]{ChP18} for more details.
	
	\subsection{Parameterized Landau--Ginzburg models}
	
	We start by recalling some definitions in ~\cite{DHKOP24}.
	
	\begin{definition}\label{def:parametrized LG}
		Let $X$ be a smooth Fano variety, and $\mathrm{Pic}(X) \cong \mathbb{Z}^{\rho(X)}$ be any choice of basis. We say a Laurent polynomial $f \in \mathbb{C}[a_1^{\pm 1}, \cdots, a_{\rho(X)}^{\pm 1}][x_1^{\pm 1}, \cdots, x_{\dim(X)}^{\pm 1}]$ is a \emph{parameterized toric Landau--Ginzburg model} if for any choice of the complexified divisor class $D=(\alpha_1, \cdots, \alpha_{\rho(X)}) \in \mathrm{Pic}(X)\otimes \mathbb{C}$, the induced parameter specialization $f_{D}=f|_{\{a_i=\exp(-\alpha_i)\}}$ is a toric Landau--Ginzburg model for the pair $(X,D)$. In particular, if $D$ is an anticanonical class, then $f_{D}$ is called an \emph{anticanonical} toric Landau--Ginzburg model.
	\end{definition}
	
	If the anticanonical toric Landau--Ginzburg model $f_{-K_X}$ of a parameterized toric Landau--Ginzburg model $f$ is standard, we call $f$ a \emph{standard} (parameterized) toric Landau--Ginzburg model.
	
	In  \cite{DHKOP24}, it is shown that the deformation of  $f_{-K_X}$ with a Calabi--Yau compactification~\eqref{equation:CCGK-compactification} is unobstructed, and a versal family is obtained by deforming the Minkowski polynomial $f_{-K_X}$ with preserving the Newton polytope of $f_{-K_X}$. One way to do this is to consider the Givental's mirror construction associated with the descriptions of Fano threefolds in \cite{CCGK16} (cf. \cite{Lee24} for toric complete intersections).	
Specifying parameters $a_i$, we can treat $f$ as a family of toric Landau--Ginzburg models.
Therefore, we regard the standard parameterized toric Landau--Ginzburg model $f$ as a versal family of deformations of the standard toric Landau--Ginzburg model $f_{-K_X}$. Since we mainly work with compactified Landau--Ginzburg models, we consider a Calabi--Yau compactification of each member of $f$. % and denote
An element of the resulting family of compactified toric Landau--Ginzburg models 
is called a \emph{deformed Landau--Ginzburg model}, and it is said to be \emph{generic} if it is generic in the analytic sense.

	\section{Proof of Theorem \ref{theorem:any rank rationality}}
\label{section:proof}
	We prove Theorem \ref{theorem:any rank rationality} in this section. The proof is essentially a case-by-case study of the monodromy of the reducible fibers of the Landau--Ginzburg models mirror to a given Fano threefold $X$.
	\subsection{The case $X$ is rational}
\label{subsection:rational}
	Throughout this section, we assume that $X$ is rational. We need to prove that its generic Landau--Ginzburg model has unipotent monodromy around every reducible fiber, whenever such fibers exist. In many cases, verifying unipotency of the monodromy is quite simple due to the following theorem by Griffiths--Landman--Grothendieck.
	
	\begin{theorem}[see \cite{Ka70}]\label{theorem:Landman}
		Let $w\colon Y\to \mathbb{C}$ be a family of projective algebraic varieties with smooth total space $Y$. Then the monodromy around a normal crossing fiber is unipotent if the multiplicity of each its component is $1$.
	\end{theorem}
	
	\begin{corollary}
    \label{corollary:snc multiplicity 1 unipotent}
		Let $(Y,w)$ be a compactified Landau--Ginzburg model mirror to $X$. If every reducible fiber of $w$ is simple normal crossings, then the monodromy is unipotent.
	\end{corollary}

If we view a compactified Landau--Ginzburg model near a reducible fiber as a semi-stable degeneration of K3 surfaces, the condition that the reducible fiber has simple normal crossings is an open condition. By Corollary \ref{corollary:snc multiplicity 1 unipotent}, once we show that a compactified anticanonical Landau--Ginzburg model has simple normal crossings reducible fibers, it follows that there exists an open analytic neighborhood in the deformation space whose deformed reducible fibers also have simple normal crossings. Hence, in such cases, the unipotency of the monodromy is an open condition (cf. Lemma~\ref{l:isomonodromy}). Except for $X_{3-2}$, we apply this argument to prove Theorem \ref{theorem:any rank rationality} for rational $X$.

	We start from the obvious case.
	\begin{lemma}\label{lem: no component}
		If $h^{1,2}(X)=0$, then Theorem \ref{theorem:any rank rationality} holds.
	\end{lemma}
	
	\begin{proof}
		By Theorem \ref{theorem:Hodge number}, for a mirror Landau--Ginzburg model of $X$, there is no reducible fiber. Moreover, for those cases, it is straightforward to check that all irreducible fibers are indeed reduced (see ~\cite{Prz13, ChP18}). Therefore, the assertion of Theorem \ref{theorem:any rank rationality} in these cases holds.
	\end{proof}
	
	\begin{lemma}\label{lem:s.n.c case}
		If $h^{1,2}(X) \geq 1$ and $X \neq  X_{3-2}$, then there is a compactified Landau--Ginzburg model $(Y,w)$ for $X$ such that every reducible fiber of $w$ is a simple normal crossings divisor. Therefore, in this case, Theorem  \ref{theorem:any rank rationality} follows from Theorem \ref{theorem:Landman}.
	\end{lemma}
	
	Before giving a proof for all cases, we see the case $X_{2-12}$ to give a taste and explain why one needs to consider a generic Landau--Ginzburg model.
	
	\begin{example}\label{eg:2-12}
The variety $X_{2-12}$ is an intersection of $3$ divisors of type $(1,1)$ in $\PP^3\times \PP^3$.
Thus, parameterized Givental's Landau--Ginzburg model can be described as
$$
\left\{
  \begin{array}{ll}
    v_1v_2v_3v_4=b' \\
    v_5v_6v_7v_8=a' \\
    v_1+v_5=1 \\
    v_2+v_6=1 \\
    v_3+v_7=1
  \end{array}
\right.
$$
in $(\CC^{*})^8$ with coordinates $v_1,\ldots,v_8$ and parameters $a',b'\neq 0$,
see, for instance,~\cite[Definition 3.3]{PSh17}.
We have $v_4=\frac{b'}{v_1v_2v_3}$, $v_8=\frac{a'}{v_5v_6v_7}$. Thus
the parameterized Givental's Landau--Ginzburg model is isomorphic to
$$
\left\{
  \begin{array}{ll}
    v_1+v_5=1 \\
    v_2+v_6=1 \\
    v_3+v_7=1
  \end{array}
\right.
$$
in $(\CC^*)^6$ with coordinates $v_1,v_2,v_3,v_5,v_6,v_7$ and superpotential
$$
v_1+v_2+v_3+\frac{b'}{v_1v_2v_3}+v_5+v_6+v_7+\frac{a'}{v_5v_6v_7}.
$$
Let us make the birational change of coordinates
$$
v_1=\frac{1}{x'+1},\ \ v_5=\frac{x'}{x'+1},\ \ v_2=\frac{1}{y'+1},\ \ v_6=\frac{y'}{y'+1},\ \ v_3=\frac{1}{z'+1},\ \ v_7=\frac{z'}{z'+1}.
$$
Then the parameterized Givental's Landau--Ginzburg model is birational to the torus~$(\CC^*)^3$ with superpotential
$$
(x'+1)(y'+1)(z'+1)\left(b'+\frac{a'}{x'y'z'}\right).
$$
By changing variables
$$
x' = x/y,\ \  y' = y/z,\ \ z' = z
$$
and scaling $a=a'/b'$, $b=b'$ together with scaling of the coordinate on the target $\CC$ by $b$, we get the following parameterized toric Landau--Ginzburg model
$$
f_a=\frac{(x+y)(y+z)(z+1)(x+a)}{xyz}.
$$
In particular, the standard toric Landau--Ginzburg model is
$$
f=f_1=\frac{(x+y)(y+z)(z+1)(x+1)}{xyz}.
$$
By following the compactification procedure in \cite{ChP18} (see also Section \ref{s:CYcptn}), regarding the parameter $a$ as a constant, one can construct a family of compactified Landau--Ginzburg models. Consider a pencil of quartic surfaces in $\PP^3$
$$
\mathcal{S}_a=\{(x+y)(y+z)(z+t)(at+x)=\lambda xyzt\}.
$$
For $\lambda \neq \infty$, we have \footnote{Here, we use the handy notation in \cite[Section 1.6]{ChP18} }
\[
\begin{aligned}
	& H_{\{x\}}\cdot S_{a,\lambda}=L_{\{x\}, \{y\}}+L_{\{x\},\{y,z\}}+L_{\{x\},\{z,t\}}+L_{\{x\},\{t\}}, \\
	& H_{\{y\}}\cdot S_{a,\lambda}=L_{\{y\}, \{x\}}+L_{\{y\}, \{z\}}+L_{\{y\}, \{t,z\}}+L_{\{y\}, \{at,x\}},\\
	& H_{\{z\}}\cdot S_{a,\lambda}=L_{\{z\}, \{y\}}+L_{\{z\}, \{x,y\}}+L_{\{z\}, \{t\}}+L_{\{z\}, \{at,x\}},\\
	& H_{\{t\}}\cdot S_{a,\lambda}=L_{\{t\}, \{x,y\}}+L_{\{t\}, \{z,y\}}+L_{\{t\}, \{z\}}+L_{\{t\}, \{x\}}. \\
\end{aligned}
\]
Thus,
\[
\begin{aligned}
	S_{a,\infty} \cdot S_{a,\lambda}=&2L_{\{x\}, \{y\}}2L_{\{y\}, \{z\}}+2L_{\{x\}, \{t\}}+2L_{\{z\},\{t\}} \\
	&+L_{\{x\},\{y,z\}}+L_{\{x\},\{z,t\}}+L_{\{y\}, \{t,z\}}+L_{\{y\}, \{at,x\}} \\
	&+L_{\{z\}, \{x,y\}}+L_{\{z\}, \{at,x\}}+L_{\{t\}, \{x,y\}}+L_{\{t\}, \{z,y\}}
\end{aligned}
\]
These are all base curves of the pencil $\mathcal{S}_{a}$ and the multiplicity of the central fiber $S_{a,0}$ along each base curve is $2$ for $L_{\{x\}, \{y\}}, L_{\{y\}, \{z\}}, L_{\{x\}, \{t\}}, L_{\{z\},\{t\}}$ and $1$ for the rest ones. After resolving the base locus, we have a unique reducible fiber at $\lambda=0$, and four singular points $[\alpha^2:\alpha:1:1/\alpha]$ lying over $\lambda=(\alpha+1)^4$ for $\alpha$ such that $\alpha^4=a$. For example, when $a=1$, then $\alpha=\pm 1, \pm \sqrt{-1}$ so that singular fibers are over $0, 16, (1+i)^4, (1-i)^4$. Note that the point $[1:-1:1:-1]$ over $\lambda=0$ is the intersection of four irreducible components. This implies that when $a=1$, the reducible fiber is not simple normal crossings.

On the other hand, if $a \neq 1$, then there are four singular fibers with isolated singularities, not lying over $\lambda =0$, and one unique reducible fiber over $\lambda=0$, that is simple normal crossings.
While the reducible fiber not being a simple normal crossings divisor for $a=1$ does not necessarily imply that the monodromy around it is non-unipotent, this can be deduced from our parameterized Landau--Ginzburg models. Let $U$ be a small neighborhood of $a=1$ in the parameter space. Consider a small disk $\mathbb{D}$ centered at $0 \in \mathbb{C}$ such that, for each $a \in U$, the map $f_a$ has at most one singular fiber with an isolated singularity over $\mathbb{D}$. Then the monodromy around $\partial \mathbb{D}$ defines a family of endomorphisms $\Phi = {\Phi_a \mid a \in U}$ of the trivial vector bundle over $U \times \mathbb{D}$ whose fiber is $H^{2}(w_a^{-1}(e))$ for $e \in \partial \mathbb{D}$.

Away from $a=1$, the monodromy $\Phi_a$ is strictly quasi-unipotent, since it is the composition of the monodromy around the simple normal crossings fiber and that around a singular fiber with one ordinary double point, the latter having non-unit eigenvalues. This implies that the coefficients of the characteristic polynomials are constant functions on $U \setminus \{1\}$, and hence are constant on all of $U$ by analytic continuation.

For the standard Landau--Ginzburg model $f=f_1$, the quartic closure  $S_{-2}$ of the initial fiber is an irreducible quartic whose singular locus is a union of two lines $L_1 \cup L_2$. There are $6$ fixed singular points and one of them is $P=L_1 \cap L_2$.
By resolving these singularities, no exceptional surface appears, but one exceptional curve $C$ does. While these three curves $L_1, L_2$ and $C$ is of multiplicity $2$, their intersection is non-empty: $L_1 \cap L_2 \cap C \neq \emptyset$. Therefore, the resulting surface is not simple normal crossings. 
	\end{example}
	
	\begin{proof}[Proof of Lemma~\ref{lem:s.n.c case}]
	If a toric Landau--Ginzburg model for a certain family is not mentioned, we take one from~\cite{ChP18} and study the pencil $\mathcal{S}$ and its fibers $S_\lambda$ in the notation of Subsection~\ref{s:CYcptn}. Here, we may abuse notation by writing the resolution of $S_\lambda$ along fixed singular points in the same way. In particular, we show that the reducible fibers are simple normal crossings divisors by applying Lemma \ref{lem:multiplicity}. This analysis must be carried out separately for each case.
	
		\begin{itemize}
			\item $X_{1-6}$. Let us take a toric Landau--Ginzburg model $$\frac{(x+z+1)(x+y+z+1)(z+1)(y+z)}{xyz}$$ from \cite[Table 1]{Prz13}. By following the compactification procedure in \cite{ChP18} (see also Section \ref{s:CYcptn}), we start with a pencil of quartic surfaces in $\mathbb{P}^3$
			\[
			\mathcal{S}=\{(x+z+t)(x+y+z+t)(z+t)(y+z)=\lambda xyzt\}.\]
			For $\lambda \neq \infty$, we have
			\[
			\begin{aligned}
				& H_{\{x\}}\cdot S_{\lambda}=2L_{\{x\}, \{z,t\}}+L_{\{x\}, \{y,z,t\}}+L_{\{x\}, \{y,z\}},  \\
				& H_{\{y\}}\cdot S_{\lambda}=2L_{\{y\}, \{x,z,t\}}+L_{\{y\}, \{z,t\}}+L_{\{y\}, \{z\}},  \\
				& H_{\{z\}}\cdot S_{\lambda}=L_{\{z\}, \{x,t\}}+L_{\{z\}, \{x,y,t\}} +L_{\{z\},\{t\}}+L_{\{z\}, \{y\}},\\
				& H_{\{t\}}\cdot S_{\lambda}=L_{\{t\}, \{x,z\}}+L_{\{t\}, \{x,y,z\}}+L_{\{t\}, \{z\}}+L_{\{t\}, \{y,z\}}.\\
			\end{aligned}
			\]
			        Thus, for $\lambda \in \mathbb{C}$,
			\[
			\begin{aligned}
				S_{\infty} \cdot S_{\lambda}=&2L_{\{x\}, \{z,t\}}+L_{\{x\}, \{y,z,t\}}+L_{\{x\}, \{y,z\}}+2L_{\{y\}, \{x,z,t\}} \\
				&+L_{\{y\},\{z,t\}}+2L_{\{y\},\{z\}}+2L_{\{t\},\{z\}}+L_{\{z\}, \{x,t\}}\\
				&+L_{\{z\}, \{x,y,t\}}+L_{\{t\}, \{x,z\}}+L_{\{t\}, \{x,y,z\}}+L_{\{t\}, \{y,z\}}.
			\end{aligned}
			\]
			These are all base curves of the pencil $\mathcal{S}$ and the multiplicity of the central fiber $S_{0}$ along each base curve is $1$ except $L_1=L_{\{x\}, \{z,t\}}$ and $L_2=L_{\{y\}, \{x,z,t\}}$.
			
			The surface $S_{\lambda}$ has seven fixed singular points,
\begin{align*}
P_1=[0:1:-1:0],\ \  P_2=[0:1:0:0],\ \  P_3=[0:0:1:-1],\ \  P_4=[1:0:-1:0],\\ P_5=[1:0:0:0],\ \  P_6=[1:0:0:-1],\ \  P_7=[1:-1:0:0].
\end{align*}
  Resolution of these singularities introduces new base curves. Among them are two quadric base curves of multiplicity $2$, arising from blow-up of $P_3$ and $P_6$ each, which we denote by $C_3$ and $C_6$. By Lemma \ref{lem:multiplicity}, the curves $L_1, L_2, C_3,$ and $C_6$ each produce an exceptional component in the central fiber. By construction, the resulting surface has simple normal crossings.

			\item $X_{1-8}$. Let us take a toric Landau--Ginzburg model $$\frac{(x+1)(y+1)(z+1)(x+y+z+1)}{xyz}$$ from \cite[Table 1]{Prz13}. In a Calabi--Yau compactification, it has no additional component  other than the (closure of) initial components in the central fiber (see ~\cite[Example 29]{Prz13}). Therefore, the resulting reducible fiber is simple normal crossings.
			\item $X_{1-9}$.  Let us take a toric Landau--Ginzburg model $$\frac{(x+y+z)(x+xz+xy+xyz+z+y+yz)}{xyz}$$ from \cite[Table 1]{Prz13}. Other than the (closure of) initial components, there will be one component appearing with multiplicity one (see ~\cite[Example 30]{Prz13}). It follows that the resulting reducible fiber is simple normal crossings.
			
			\item $X_{1-14}$. Let us take a toric Landau--Ginzburg model $$\frac{(x+1)^2(y+1)^2}{xyz}+z$$ from~\cite[Table 1]{Prz13}. It is mutationally equivalent to $(x+1)(y+1)\left(\frac{1}{xyz}+z\right)$. By~\cite[Theorem 1.2]{PS15} its Calabi--Yau compactification has $3$ components of the unique reducible fiber, and, thus, no additional component other than the (closure of) initial components in the central fiber. Since the compactifications of the initial components are simple normal crossings, this shows that the reducible fiber is simple normal crossings.
			\item $X_{2-4}$: The initial fiber over $\lambda=-7$ is $S_{-7}=H_1+H_2+Q$ where $Q$ is an irreducible quadric surface. The intersection is given by $L_1=H_1 \cap H_2$, $C_1=H_1 \cap Q$, and $C_2=H_2 \cap Q$, where $C_1$ is an irreducible conic and $C_2$ is a union of two lines $L_2 \cup L_3$. Note that $L_1, L_2, L_3$ are singular locus of $S_{-7}$ that is contained in the base locus (with multiplicity $2$). Also, there are four singular points $P_1, P_2, P_3, P_4$, where $P_1=L_1 \cap L_2$, $P_2=L_2 \cap L_3$, $P_3=L_1 \cap L_3$, and $P_4$ is an ordinary double point.
			
			By resolving these singularities, we obtain two exceptional surfaces $E_1$ and $E_3$ from the blow-up of $P_1$ and $P_3$, respectively. Let $C_1 =S_{-7} \cap E_1$ and $C_3=S_{-7} \cap E_3$. Also, when blowing up $P_2$, no exceptional surface appears, but one base curve $C_2$ appears. Then $S_{-7}+E_1+E_2$ becomes a simple normal crossings divisor. The base locus is now given by a simple normal crossings union of $17$ rational curves where six of them $L_1, L_2, L_3, C_1, C_2, C_3$ lying in the singular locus of $S_{-7}+E_1+E_2$. They all have multiplicity $2$, hence each of them produces exceptional surface. Therefore, the resulting surface has simple normal crossings.
			
			\item $X_{2-7}$. The initial fiber over $\lambda=-5$ is $S_{-5}=Q_1+Q_2$, where $Q_1$ is a smooth quadric and $Q_2$ is a singular quadric with a single isolated singularity, which is a fixed singular point to be resolved and is also away from the intersection locus. The intersection locus is a union of two conics $C_1$ and $C_2$ and $C_1 \cap C_2=\{P_1, P_2\}$ are the singular locus. There are other singularities which does not affect our calculation.
			
			By resolving singular points $P_1$ and $P_2$, no exceptional surface appears, but two base curves $C_3$ and $C_4$ does, respectively. Among other base curves, $C_1, C_2, C_3, C_4$ have multiplicity $2$, each of which produces an exceptional surface in the pencil.  Therefore, the resulting surface has simple normal crossings.
 
 			\item $X_{2-9}$. The initial fiber over $\lambda=-3$ is $S_{-3}=H_1+H_2+Q$, where $Q$ is a smooth quadric. The intersection $H_1 \cap Q$ is a conic, and $H_2 \cap Q$ is a union of two lines, denoted by $L_1$ and $L_2$. Note that $P=L_1 \cap L_2$ is a singular point of type $\mathbb{D}_4$.
			
			By resolving the singular point $P$, no exceptional surface appears, but one base curve, denoted by $C_1$, does. Note that the strict transforms of $L_1$ and $L_2$ are separated by $C_1$. Among the other base curves, $L_1$, $L_2$, and $C_1$ each have multiplicity 2, and each produces an exceptional surface in the pencil. Therefore, the resulting surface has simple normal crossings.

			\item $X_{2-10}$. There are two singular fibers.
			\begin{enumerate}
				\item At $\lambda=-4$, the initial fiber is $S_{-4}=H+{S}$, where $S$ is a cubic surface with two singular points $P_1$ (ordinary double point) and $P_2$ (type $
				\mathbb{A}_2)$. The intersection is a irreducible cubic which contains $P_1$.
				By resolving these singular points, only one exceptional surface $E_1$ for $P_1$ appears. Then $S_{-4}+E_1$ has simple normal crossings. Since it turns out that there is no base curves with higher multiplicity, the resulting surface in the pencil is $S_{-4}+E_1$.

				\item At $\lambda=-5$, the initial fiber is $S_{-5}=Q_1+Q_2$, where $Q_1$ and $Q_2$ are smooth quadric surfaces. The intersection is an irreducible conic. In this case, we only have two components, so there is nothing much to check.
				
			\end{enumerate}
			\item $X_{2-12}$. See Example~\ref{eg:2-12} for the proof.
			\item $X_{2-13}$. The initial fiber $S_{-3}$ is an irreducible quartic surface with $6$ isolated singularities.			
			By resolving these singularities, one exceptional surface $E$ appears so the strict transform of $S_{-3}$ becomes $S_{-3}+E$. The intersection of $S_{-3} \cap E$ becomes a base curve $C$ with multiplicity $2$. Among all the base curves, this is a unique one with higher multiplicity, which produces an exceptional surface in the pencil. Therefore, the resulting fiber has simple normal crossings.
			
			\item $X_{2-15}$. The initial fiber over $\lambda=-1$ is $S_{-1}=H+{S}$, where ${S}$ is an irreducible cubic surface with three isolated singularities away from ${S} \cap H$. The intersection is a union of three lines $L_1, L_2, L_3$. Let $P=L_1 \cap L_2$. By resolving these singularities, no exceptional surface appears, but one base curve $C$ (corresponding to $P$) does. Since both ${S}$ and $H$ are smooth on $P$, their strict transform pass through the base curve $C$. Among all the base curves, $C, L_1, L_2$ (forming a chain) have multiplicity $2$, each of which produce an exceptional surface in the pencil. Therefore the resulting surface has simple normal crossings.
			
			\item $X_{2-16}$. The initial fiber over $\lambda=-2$ is $S_{-2}=H_1+H_2+Q$, where $Q$ is an irreducible quadric. The intersection $H_1 \cap H_2 \cap Q$ is a set of two points which does not belong to the global singular locus. There are no base curves with higher multiplicities, thus in the pencil, the resulting surface looks like a configuration of American football ball (see Figure \ref{fig:American}). Therefore, the resulting surface has simple normal crossings.
			\begin{figure}
			\begin{tikzpicture}[every node/.style={circle,draw,fill=black,inner sep=0.5pt},
				edge/.style={line width=0.5pt}]
				% vertices
				\node (A) at (0,0) {};
				\node (B) at (3,0) {};
				
				% three edges: outer up, inner (slightly curved), outer down
				\draw[edge] (A) to[bend left=50]  (B);
				\draw[edge] (A) to[bend left=10]   (B);
				\draw[edge] (A) to[bend right=50] (B);
			\end{tikzpicture}
			\caption{American football ball}
			\label{fig:American}
			\end{figure}
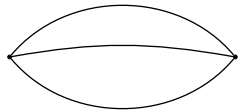
			
			\item $X_{2-17}$. The initial fiber over $\lambda=-2$ is $S_{-2}=H+{S}$,
			where ${S}$ is an irreducible cubic surface with two isolated singularities that lie not over the intersection locus. After a resolution, there are no base curves with higher multiplicities, thus in the pencil, the resulting surface has simple normal crossings.
			
			\item $X_{2-18}$. The computation is almost identical to $X_{2-16}$.
		
			\item $X_{2-19}$. The computation is almost identical to $X_{2-16}$.
			
			\item $X_{2-23}$. The computation is almost identical to $X_{2-17}$.

			\item $X_{2-25}$. The computation is almost identical to $X_{2-10}$ at $\lambda=-5$.

			\item $X_{2-28}$. The computation is almost identical to $X_{2-17}$.

			\item $X_{3-2}$.  We treat this case separately in Lemma~\ref{lem:3-2} below, because the reducible fiber of the Calabi--Yau compactification of its standard Landau--Ginzburg model is not simple normal crossings. In fact, the fiber of Calabi--Yau compactification
over $\lambda=0$ looks like a configuration of American football ball, where all $4$ components $\mathsf{S}_1, \mathsf{S}_2, \mathsf{S}_3, \mathsf{S}_4$ meet at two distinct points, and $\mathsf{S}_i \cap \mathsf{S}_{i+1} = \mathbb{P}^1$ for $i=1,2,3,4$ and $\mathsf{S}_5=\mathsf{S}_1$. 
			
			\item $X_{3-3}$. We first observe that the standard Landau--Ginzburg model in ~\cite[Section 3.3]{ChP18} has a reducible fiber which is not simple normal crossings.
			
			The initial fiber $S_{-4}$ is an irreducible quartic surface with three singularities. There are two base curves $L_1$ and $L_2$ lying in $S_{-4}$ and $L_1 \cap L_2$ is one of these three singularities. Resolving this singularity, one exceptional surface $E$ appears so that the strict transform of $S_{-4}$ becomes $S_{-4}+E$. Also, strict transforms of $L_1$ and $L_2$ are separated and intersect with  $E$ at a point and  $L_1$ and $L_2$ are the base curves with multiplicity $2$, hence produces one exceptional surface each. However, the intersection $S_{-4} \cap E$ is a union of a non-reduced conic (double line) and a line.  Therefore, the resulting fiber is \emph{not} simple normal crossings.

Note that the variety $X_{3-3}$ can be described as a divisor of type $(1,1,2)$ on $\PP^1\times \PP^1\times \PP^2$.
    The standard birational change of variables and shifting of the coordinate on the base for Givental's Landau--Ginzburg model (cf. Example~\ref{eg:2-12}) leads to the parameterized toric Landau--Ginzburg model
			$$
			(p+q+r+1)\left(\frac{a}{p}+\frac{b}{q}+\frac{c(p+q+r+1)}{r}\right) \in\CC_{a,b,c}[p^{\pm 1},q^{\pm 1}, q^{\pm 1}]
			$$
			for $a,b,c \neq 0$. Changing coordinates and parameters and scaling coordinate on the base, one may consider the parameterized toric Landau--Ginzburg model
			$$f_{a,b}= (x+y+z+1)\left(\frac{a}{x}+\frac{b}{y}+\frac{x+y+z+1}{z}\right) \in\CC_{a,b}[x^{\pm 1},y^{\pm 1}, z^{\pm 1}].
			$$
			By applying the compactification procedure in \cite{ChP18} (see also Section \ref{s:CYcptn}), regarding $a,b\neq 0$ as constants, one can construct a family of compactified Landau--Ginzburg models. Consider a pencil of quartic surfaces in $\mathbb{P}^3$
			\[
			\mathcal{S}_{a,b}=\{(x+y+z+t)(ayzt+bxzt+xy(x+y+z+t))=\lambda xyzt\}.\]
			For $\lambda \neq \infty$, we have
			\[
			\begin{aligned}
				& H_{\{x\}}\cdot S_{a, b,\lambda}=L_{\{x\}, \{y,z,t\}}+L_{\{x\}, \{y\}}+L_{\{x\}, \{z\}}+L_{\{x\}, \{t\}},  \\
				& H_{\{y\}}\cdot S_{a, b,\lambda}=L_{\{y\}, \{x,z,t\}}+L_{\{y\}, \{x\}}+L_{\{y\}, \{z\}}+L_{\{y\}, \{t\}},  \\
				& H_{\{z\}}\cdot S_{a, b,\lambda}=2L_{\{z\}, \{x,y,t\}}+L_{\{z\}, \{x\}}+L_{\{z\}, \{y\}},\\
				& H_{\{t\}}\cdot S_{a, b,\lambda}=2L_{\{t\}, \{x,y,z\}}+L_{\{t\}, \{x\}}+L_{\{t\}, \{y\}}. \\
			\end{aligned}
			\]
			Thus,
			\[
			\begin{aligned}
				S_{a,b,\infty} \cdot S_{a,b,\lambda}=&L_{\{x\}, \{y,z,t\}}+L_{\{y\}, \{x,z,t\}}+2L_{\{z\}, \{x,y,t\}}+2L_{\{t\}, \{x,y,z\}} \\
				&+2L_{\{x\},\{t\}}+2L_{\{y\},\{t\}}+2L_{\{x\},\{y\}}+2L_{\{x\},\{z\}}+2L_{\{y\},\{z\}}.
			\end{aligned}
			\]
			These are all base curves of the pencil $\mathcal{S}_{a,b}$ and the multiplicity of the central fiber $S_{a,b,0}$ along each base curve is $1$ except for $L_{\{z\}, \{x,y,t\}}$ and $L_{\{t\}, \{x,y,z\}}$. For simplicity, we write $L_1=L_{\{z\}, \{x,y,t\}}$ and $L_2=L_{\{t\}, \{x,y,z\}}$.
			
			We have two initial components $S_{a,b,0}=H+S$, where $H=\{x+y+z+t=0\}$ and $S=\{ayzt+bxzt+xy(x+y+z+t)=0\}$. Note that $S$ is an irreducible cubic with 4 isolated singularities $[0:0:1:0]$, $[0:0:0:1]$, $[0:1:0:-1]$, and $[1:0:0:-1]$. They are all fixed singular points, and there are $4$ more singular points, that is $[1:-1:0:0]$, $[0:1:-1:0]$, $[1:0:-1:0]$, and $[0:0:1:-1]$.
			
			Moreover, $H \cap S=L_1 \cup L_2 \cup C_{a,b}$ where
			\[
			C_{a,b}=\{x+y+z+t=ay+bx=0\}.\]
			Therefore, if $a \neq b$, then there is no triple intersection point, and in this case one has $[0:0:1:-1]=L_1 \cap L_2$.
			
			From now on, we assume that $a \neq b$. By resolving fixed singular points, no exceptional component appears, but $L_1$ and $L_2$ are separated out by the exceptional line, denoted by $C_3$. Then the strict transform of $S_{a,b,\lambda}$ consists of two components whose intersection locus is the union of $C_3$, $L_1$, $L_2$, and $C_{ab}$, where $S \cap C_{ab}=L_1 \cap L_2= \emptyset$ and it forms a wheel. Since the multiplicity of $L_1$ and $L_2$, after blowing them up, we get $2$ extra components. Therefore, the resulting surface has simple normal crossings.
			
			If $a=b$, the exceptional line $C_3$ becomes a double line, and the situation then resembles the compactification of the anticanonical Landau--Ginzburg model. In fact, one can check that there is no collision unlike Example \ref{eg:2-12}, thus the similar argument can be applied to show that the monodromy for $a=b$ is unipotent.
			\item $X_{3-4}$. The computation is almost identical to $X_{2-13}$.
			%		(Remark: this is a blow up of $X_{2-18}$ whose boundary complex of the central fiber is American football ball configuration).
			\item $X_{3-6}$. The computation is almost identical to $X_{2-17}$.
			\item $X_{3-7}$. The computation is almost identical to $X_{2-17}$.
			\item $X_{3-9}$. The initial fiber over $\lambda=-2$ is $S_{-2}=H+{S}$, where $S$ is an irreducible cubic surface with four isolated singularities $P_1, P_2, P_3, P_4$. The intersection $H \cap {S}$ is a union of three lines $L_1, L_2, L_3$ with no triple point. Also, $P_i$ lies in $L_i \setminus L_{j} \cup L_{k}$ for $\{i,j,k\}=\{1,2,3\}$, and $P_4=L_1 \cap L_2$. Resolving these isolated singularities, no exceptional surfaces and base curves appear.
Among all the base curves, $L_1$ and $L_2$ are the ones with multiplicity $2$, each of which produces an exceptional surface in the pencil. Therefore, the resulting surface has simple normal crossings.
			\item $X_{3-14}$. The computation is almost identical to $X_{2-17}$.
			\item $X_{4-1}$. The computation is almost identical to $X_{2-17}$.
			\item $X_{4-2}$. The computation is almost identical to $X_{2-17}$.
		\end{itemize}
	\end{proof}
	
	\begin{lemma}\label{lem:3-2}
		Let $X=X_{3-2}$. A generic Landau--Ginzburg model $(Y,w)$ has unipotent monodromy around every reducible fiber.
	\end{lemma}
	\begin{proof}
		Let us consider the standard toric Landau--Ginzburg model for $X_{3-2}$ %of Givental's type
		$$
		f=(p+q+1)\left(\frac{1}{r}+\frac{r}{q}+\frac{(p+q+1)^2}{pq} \right).
		$$
        One can easily see (cf.~\cite[Example 3.11]{DHKOP24}), that Minkowski Laurent polynomials that are deformations of $f$
and have the same Newton polytope are
		$$
		(p+q+1)\left(\frac{1}{r}+\frac{cr}{q}+b(p+q+1)\left(\frac{ap+q+1}{pq}\right) \right)
		$$
		for complex parameters $a$, $b$, $c$ with $a, b, c\neq 0$.
Since the family of Landau--Ginzburg models of $X$ is two-parameter, each of these Landau--Ginzburg
model has a torus chart that gives a Minkowski Laurent polynomial, and for different
models these polynomials differ (because they have different periods). By~\cite{Prz16}
these polynomials have Calabi--Yau compactifications, and
these compactifications form the deformation family of Landau--Ginzburg models.

After toric change of variables one gets the parameterized toric Landau--Ginzburg model
		$$%f_X^{GHV}= 	
		f_{a,b}=(x+y+1)\left(\frac{1}{z}+\frac{z}{y}+b(x+y+1)\left(\frac{ax+y+1}{xy}\right) \right).
		$$
        In particular, $f=f_{1,1}$ after the same coordinate change.
		By following the compactification procedure in \cite{ChP18} (see also Section \ref{s:CYcptn}), regarding $a$ and $b$ as constants, one can construct a family of compactified Landau--Ginzburg models. Consider a pencil of quartic surfaces in $\mathbb{P}^3$
		\[
		\mathcal{S}_{a,b}=\{(x+y+t)(x(ty+z^2)+bz(x+y+t)(ax+y+t)=\lambda xyzt\}.\]
		For $\lambda \neq \infty$, we have
		\[
		\begin{aligned}
			& H_{\{x\}}\cdot S_{a, b,\lambda}=3L_{\{x\}, \{y,t\}}+L_{\{x\},\{z\}}, \\
			& H_{\{y\}}\cdot S_{a, b,\lambda}=C_1+L_{\{y\}, \{x,t\}}+L_{\{y\}, \{z\}},\\
			& H_{\{z\}}\cdot S_{a, b,\lambda}=L_{\{z\}, \{x,y,t\}}+L_{\{z\}, \{x\}}+L_{\{z\}, \{y\}}+L_{\{z\}, \{t\}},\\
			& H_{\{t\}}\cdot S_{a, b,\lambda}=C_2+L_{\{t\}, \{x,y\}}+L_{\{t\}, \{z\}}, \\
		\end{aligned}
		\]
		where $C_1=\{xz+b(x+t)(ax+y+t)=0\}$ and $C_2=\{xz+b(x+y)(ax+y)=0\}$.
		Thus,
		\[
		\begin{aligned}
			S_{a,b,\infty} \cdot S_{a,b,\lambda}=&3L_{\{x\}, \{y,t\}}+L_{\{z\}, \{x,y,t\}}+L_{\{y\}, \{x,t\}}+L_{\{t\}, \{x,y\}}+C_1+C_2 \\
			&+ 2L_{\{z\},\{t\}}+2L_{\{x\},\{z\}}+2L_{\{y\},\{z\}}.
		\end{aligned}
		\]
		These are all base curves of the pencil $\mathcal{S}_{a,b}$ and the multiplicity of the central fiber $S_{a,b,0}$ along each base curve is $1$ except for $L_{\{x\}, \{y,t\}}$ whose multiplicity is $2$. For simplicity, we denote this curve by $L$.
		
		We have two initial components $S_{a,b,0}=H+{S}$ where $H=\{x+y+t=0\}$ and ${S}=\{x(ty+z^2)+bz(x+y+t)(ax+(y+t))=0\}$ whose intersection is the union of the line $L$ and conic ${C}_{a,b}=\{x+y+t=ty+z^2=0\}$. Here, ${S}$ is an irreducible quadric with isolated singularities at the intersection $L \cap C$. There are $4$ fixed singular points $[0:0:1:0]$, $[0:1:0:-1]$, $[1:-1:0:0]$, and $[1:0:0:-1]$, which are double everywhere and do not lie in the intersection $L \cap C_{a,b}$.
		
		By Lemma \ref{lem:multiplicity}, after blowing up base loci, there are two extra components appearing in the central fiber. Therefore, the resulting surface is a configuration of American football ball as described in the proof of Lemma \ref{lem:s.n.c case}, but with four components. Moreover, the base loci are independent of the parameters $a$ and $b$. Therefore by Lemma \ref{l:isomonodromy}, the family of compactified Landau--Ginzburg models constructed above are isomonodromic around $0 \times \mathbb{D}^2_{a,b}$.
		
		Let $f=f_{1,1}$ be the anticanonical toric Landau--Ginzburg model for $X_{3-2}$.
		Since the pencil given by $f$ is birational to Givental's Landau--Ginzburg model,
		$I_f$ is equal to its periods,
		which, by~\cite[Theorem 0.1]{Gi97} (see also discussion after Corollary 0.4 in~\cite{Gi97})
		is equal to $\widetilde{I}_0^{X_{3-2},0}$. Thus one can write down its Picard--Fuchs equation by~\cite[Theorem 1.2]{Prz07b}.
		(This equation, in another coordinates, is written down in~\cite{fanosearch}.)
		One can check (see again~\cite{fanosearch}) that monodromy around its unique reducible fiber is unipotent, so by Lemma~\ref{l:isomonodromy}
		it is unipotent for generic Landau--Ginzburg model for $X_{3-2}$.
	\end{proof}

	\subsection{The case $X$ is not rational}
\label{subsection:non-rational}
	In this section, we assume that $X$ is not rational. In this case, there are two types of Landau--Ginzburg models; one whose reducible fiber has multiple components and the other whose reducible fiber has no multiple component but not simple normal crossings. For the former case, the strictness of quasi-unipotency follows from Kulikov’s classification~\cite{Kul77}. For the latter case, we compute the monodromy directly and show that the parameterized Landau--Ginzburg model yields an isomonodromic deformation around the reducible fiber.
	
	\begin{proposition}\label{prop:quasi-unipotency}
		Let $\pi\colon\mathfrak{X} \to \Delta$ be a degenerating family of $K3$-surfaces, where $\mathfrak{X}$ is smooth, $K_{\mathfrak{X}}=0$, and $\mathfrak{X}_0=\pi^{-1}(0)$ is non-reduced. Then the monodromy on $H^2(\mathfrak{X}_t)$ around $0$ is strictly quasi-unipotent.
	\end{proposition}
	
	\begin{proof}
		Suppose that the monodromy is unipotent. By ~\cite[Theorem 5.3]{Huy16}, the family $\pi:\mathfrak{X} \to \Delta$ is a birational modification to Kulikov's model. Since the original family and Kulikov's model are $K$-equivalent (because canonical classes
of both of them are trivial), by~\cite[Theorem 1]{Ka08} they should differ by flops only. However, flops does not change multiplicity of the components in the central fiber. This leads to a contradiction.
	\end{proof}

	\begin{lemma}\label{lem:multiple comp}
        Let $X=X_{m-n}$, where
        $m-n$ is one of the following: $1-1$, $1-2$, $1-3$, $1-4$, $1-5$, $1-11$, $1-12$, $2-2$, $2-3$, $2-6$,
        $2-8$, $3-1$. 
Then Theorem~\ref{theorem:any rank rationality} holds for~$X$.
    \end{lemma}

    \begin{proof}
To prove the assertion of the lemma we find a parameterized toric Landau--Ginzburg model that has a factor of multiplicity greater than $1$. Then the claim follows from Theorem~\ref{theorem:standard are toric}, Proposition \ref{prop:quasi-unipotency}, and Remark \ref{remark:mutation equivalence}.
Let us do it case-by-case.

	\begin{itemize}
		\item $X_{1-1}$: $\frac{(x+y+z+1)^6}{xyz}$, see~\cite[Table 1]{Prz13}.
		\item $X_{1-2}$: $\frac{(x+y+z+1)^4}{xyz}$, see~\cite[Table 1]{Prz13}.
		\item $X_{1-3}$: $\frac{(x+1)^2(y+z+1)^3}{xyz}$, see~\cite[Table 1]{Prz13}.
		\item $X_{1-4}$: $\frac{(x+1)^2(y+1)^2(z+1)^2}{xyz}$, see~\cite[Table 1]{Prz13}.
		\item $X_{1-5}$: $\frac{(1+x+y+z+xy+xz+yz)^2}{xyz}$, see~\cite[Table 1]{Prz13}.
		\item $X_{1-11}$: The polynomial $\frac{(x+y+1)^6}{xy^2z}+z$ from~\cite[Table 1]{Prz13} is mutationally equivalent to $(x+y+1)^3\left(\frac{1}{xy^2z}+z\right)$.
%, so the assertion follows from Remark~\ref{remark:mutation equivalence}.
		\item $X_{1-12}$: The polynomial $\frac{(x+y+1)^4}{xyz}+z$ from~\cite[Table 1]{Prz13} is mutationally equivalent to $(x+y+1)^2\left(\frac{1}{xyz}+z\right)$. %, so the assertion follows from Remark~\ref{remark:mutation equivalence}.
		\item $X_{2-1}$:
	The variety $X_{2-1}$ has a toric degeneration to a hypersurface
	in a certain toric variety, see~\cite[H.18]{CCGK16}.
    The standard birational change of variables and shifting the coordinate on the base for Givental's Landau--Ginzburg model (cf. Example~\ref{eg:2-12}) leads to the parameterized toric Landau--Ginzburg model
        \[(p+1)\left(a+\frac{b(q+r+1)^6}{pq^2r}\right)
        \]
        with $a,b\neq 0$. Making the change of variables
$$
p= \frac{(y+z+1)^3}{xy^2z} ,\ \ q=y,\ \ r=z
$$
one gets the parameterized toric Landau--Ginzburg model
        $$
        a+(y+z+1)^3\left(bx+\frac{a+bx(y+z+1)^3}{xy^2z}\right).
$$
The fiber of Calabi--Yau compactification of the initial polynomial over $a$ is not reduced.
		\item $X_{2-2}$: 	The variety $X_{2-2}$ can be described as a hypersurface
	in a certain toric variety, see~\cite[H.19]{CCGK16}.
    The standard birational change of variables and shifting the coordinate on the base for Givental's Landau--Ginzburg model (cf. Example~\ref{eg:2-12}) leads to the parameterized toric Landau--Ginzburg model
		\[\left(1+x+z+\frac{b}{yz}\right)^2\left(y+\frac{ay^2z^2}{x}\left(1+x+z+\frac{b}{yz}\right)\right)
		\]
		for $a , b \neq 0$. 
		\item $X_{2-3}$:
	The variety $X_{2-3}$ has a toric degeneration to a hypersurface
	in a certain toric variety, see~\cite[H.20]{CCGK16}.
    The standard birational change of variables and shifting the coordinate on the base for Givental's Landau--Ginzburg model (cf. Example~\ref{eg:2-12}) leads to the parameterized toric Landau--Ginzburg model
		\[
        a(x+1)+\frac{b(x+1)(y+z+1)^4}{xyz}
        \]
        for $a, b \neq 0$. It is mutationally equivalent to
        \[
        f=a+ax(y+z+1)^2+\frac{b(x(y+z+1)^2+1)(y+z+1)^2}{xyz}.
        \]
        In particular, its fiber over $a$ is not reduced.
        (Another way to check it is given by the calculations from~\cite[Lemma 2.3.7]{ChP18}.)
		\item $X_{2-6}$:
	The variety $X_{2-6}$ can be described as a divisor of type $(2,2)$ in $\PP^2\times \PP^2$.
    The standard birational change of variables and shifting the coordinate on the base for Givental's Landau--Ginzburg model (cf. Example~\ref{eg:2-12}) leads to the parameterized toric Landau--Ginzburg model
		\[\left(y+\frac{b}{x}\right)\left(z+\frac{a}{yz}+x+1\right)^2
		\]
		for $a , b \neq 0$. 
		\item $X_{2-8}$:
	The variety $X_{2-8}$ can be described as a hypersurface
	in a certain toric variety, see~\cite[H.25]{CCGK16}.
    The standard birational change of variables and shifting the coordinate on the base for Givental's Landau--Ginzburg model (cf. Example~\ref{eg:2-12}) leads to the parameterized toric Landau--Ginzburg model
		\[
		b\frac{(x+y+z+1+axyz)^2}{xyz}.
		\]
		for $a , b \neq 0$. 
		\item $X_{3-1}$:
	The variety $X_{3-1}$ can be described as a hypersurface
	in a certain toric variety, see~\cite[H.54]{CCGK16}.
    The standard birational change of variables and shifting the coordinate on the base for Givental's Landau--Ginzburg model (cf. Example~\ref{eg:2-12}) leads to the parameterized toric Landau--Ginzburg model

		\[
		c\frac{(x+y+z+1+\frac{ayz}{x}+byz)^2}{yz}.
		\]
		for $a , b,c \neq 0$. 
	\end{itemize}
\end{proof}

	For the rest of this section, we treat the remaining cases 	$X_{1-7}, X_{1-13}, X_{2-5}$, and $X_{2-11}$ separately. These are the cases in which the reducible fiber has no higher multiplicity components, yet it is not a simple normal crossings divisor. To check that the monodromy is not unipotent, we write down the Picard--Fuchs equation for the corresponding mirror Landau--Ginzburg model, from which we compute the monodromy. Note that this method works for anticanonical Landau--Ginzburg models. Therefore, to apply it to the rank-$2$ cases $X_{2-5}$ and $X_{2-11}$, one would need to write down a parameterized version of the Picard--Fuchs equation.  Unfortunately, the authors are not familiar with this approach. Instead, we provide a more geometric argument to describe the singular locus of the parameterized Landau--Ginzburg model and show that it is isomonodromic around the reducible fiber. 

\begin{lemma}[{cf.~\cite[Theorem 3.3]{KP09}}]
\label{lemma:1-7}
Let $X$ be a general variety from the family $X_{1-7}$, i.\,e. a section of the Grassmannian $G(2,6)$ by five general
hyperplanes.
Consider a standard Landau--Ginzburg model $(Y,w)$ of $X$.
Then $w$ has a unique reducible fiber, and the monodromy around it is not unipotent. Therefore, Theorem \ref{theorem:any rank rationality} holds for $X$.
\end{lemma}

\begin{proof}
By~\cite[Example17]{Prz13}, the standard toric Landau--Ginzburg model
$$
f=5+\frac{(x+y+z+1)^2}{x} + \frac{(x + y + z + 1)(y + z + 1)(z + 1)^2}{xyz}
$$
(cf.~\cite[Example 2.7]{ILP13})
is birational to Givental's Landau--Ginzburg model for $X$, and $I_f$ is its period. On the other hand,
by~\cite[2.2 and Theorem 6.1.1(2)]{Prz07a}, see also~\cite[5.8]{Go05},
the period of the Givental's Landau--Ginzburg model
is a solution of the differential equation
$$
D^3-t(2D+1)(13D^2+13D+4)-3t^2(D+1)(3D+4)(3D+2)\in \CC\left[t,\frac{\partial}{\partial t}\right],
$$
where $D=t\frac{\partial}{\partial t}$.
One can easily see that a Calabi--Yau compactification of $f$ has a unique reducible fiber over $4$.
Eigenvalues of monodromy around $4$ for this equation are not units (this can be seen directly,
see also~\cite{fanosearch}), which gives the assertion of the lemma.
\end{proof}

\begin{lemma}[{cf.~\cite[Theorem 3.3]{KP09}}]
\label{lemma:1-13}
Let $X$ be a general variety from the family $X_{1-13}$, i.\,e. a smooth cubic threefold.
Consider a standard Landau--Ginzburg model $(Y,w)$ of $X$.
Then $w$ has a unique reducible fiber, and the monodromy around it is not unipotent. Therefore, Theorem \ref{theorem:any rank rationality} holds for $X$.
\end{lemma}

\begin{proof}
By~\cite[Proposition 11]{Prz13} and~\cite[Theorem 2.2]{ILP13}, one can choose a standard toric Landau--Ginzburg model
$$
f=\frac{(x+y+1)^3}{xyz}+z
$$
for $X$, and its Calabi--Yau compactification has a unique reducible fiber (over $0$), see~\cite{KP09}. This fiber consists of $6$ smooth components
forming ``open book configuration'': $3$ components intersect by one smooth curve, while $3$ others
does not intersect each other, but intersect transversally $3$ first ones (as well as their intersection curve), see Figure \ref{fig:openbook}.
\begin{figure}
	\begin{tikzpicture}[scale=1]
		\draw (0,-1) -- (0,3) -- (-2,1) -- (-2,-2) -- cycle;
		\draw (0,-1) -- (0,3) -- (3,3) -- (3,-1) -- cycle;
		\draw (0,2) -- (-1,1) -- (2,2) -- cycle;
		\draw (0,1) -- (-1,0) -- (2,1) -- cycle;
		\draw (0,0) -- (-1,-1) -- (2,0) -- cycle;
		\draw (0,3) -- (2,-2) -- (0,-1) -- cycle;
		\draw (0,2) -- (0.59,1.52);
		\draw (0,1) -- (0.95, 0.65) ;
		\draw (0,0) -- (1.3, -0.23);
	\end{tikzpicture}
\caption{Open Book Figure}
\label{fig:openbook}
\end{figure}
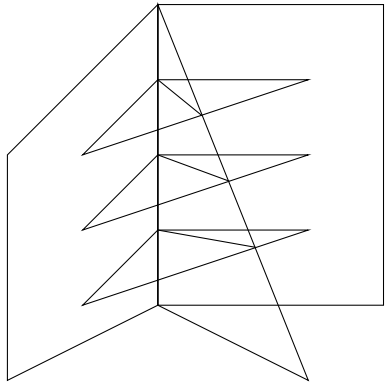
Thus, Corollary~\ref{corollary:snc multiplicity 1 unipotent} is not applicable here.

One can easily see that
$$
I_f=\sum \frac{(2k)!(3k)!}{(k!)^5}t^{2k}.
$$
According to~\cite[Page 1]{Prz07b} and~\cite[Section 2]{Prz08}, see also~\cite[5.8]{Go05}, the series $I_f$
is a solution of a differential equation given by the differential operator
$$
D^3-12t^2(3D+2)(3D+4)(D+1)\in \CC\left[t,\frac{\partial}{\partial t}\right],
$$
where $D=t\frac{\partial}{\partial t}$, which can be directly checked as well. By expanding this operator and normalizing the coefficient of $D^3$, we obtain
$$
D^3-\frac{12\cdot 27t^2}{1-108t^2}D^2-\frac{12\cdot 26t^2}{1-108t^2}D-\frac{12\cdot 8t^2}{1-108t^2}
$$
Since we are interested in the monodromy at $0$, we change the coordinate to $z=1/t$. Up to sign, the equation becomes
$$
\theta^3+\frac{12\cdot 27}{z^2-108}\theta^2-\frac{12\cdot 26}{z^2-108}\theta+\frac{12\cdot 8}{z^2-108} \in \CC\left[z,\frac{\partial}{\partial z}\right],
$$
where $\theta=z\frac{\partial}{\partial z}=-D$.

Let $f_i(z)$ denote the coefficient of $\theta^{i}$. Then the eigenvalues of the monodromy at $0$ is given by $\exp(2\pi\sqrt{-1}\alpha)$'s where $\alpha$ runs over the roots of the following polynomial
$$
y^3+f_2(0)y^2+f_1(0)y+f_0(0) = (y-1)\left(y-\frac{2}{3}\right)\left(y-\frac{4}{3}\right).
$$
Therefore, the monodromy around the central fiber is non-unipotent. (See~\cite[Theorems IV.4.1 and IV.4.2]{CL55}, see also \cite[Theorem 3]{Ve20}; cf.~\cite{fanosearch}, \cite{DHNT20} for the monodromy computation.)
\end{proof}

For the rank-two case, we require the following criterion for isomonodromy.
\begin{lemma}\label{l:isomonodromy}
	Let $g:X \to \mathbb{C}$ be a proper holomorphic map and suppose $0 \in \CC$ is a singular value. Let $(\mathbb{D},o)$ be an analytic disc centered at $o$. Suppose we have a deformation of $g$ over $(\mathbb{D},o)$, $G:\widetilde{X} \to \mathbb{C} \times \mathbb{D}$ such that
	\begin{enumerate}
		\item $\widetilde{X}$ is smooth;
		\item $G|_{G^{-1}(o)}=g$;
		\item $G$ is proper and smooth away from the discriminant locus;
		\item  There exists a smooth function $\eta:\mathbb{D} \to \mathbb{C}$ with $\eta(o)=0$ such that the graph of $\eta$ belongs to the discriminant locus of $G$ and it is diffeomorphic to  $\mathbb{D}$.
	\end{enumerate}
	Then $\pi_\mathbb{D} \circ G:\widetilde{X} \to \mathbb{D}$ is an isomonodromic deformation of $g:X \to \mathbb{C}$ around $0 \in \mathbb{C}$ where $\pi_{\mathbb{D}}$ is the projection to $\mathbb{D}$.
\end{lemma}
\begin{proof}
	By assumption, for a small neighborhood $U$ of $0 \in \mathbb{C}$, the restriction of $G$ over $U\times \mathbb{D} \setminus \Gamma(\eta)$ is a proper submersion. Therefore, we apply Ehresmann's theorem to show that fibers $G^{-1}(p,s)$ for
$$
(p,s) \in U\times \mathbb{D} \setminus \Gamma(\eta)
$$
are diffeomorphic.
	
	Fix $(p,s) \in U \times \mathbb{D}  \setminus \Gamma(\eta)$ and a loop $\Omega_s$ around $(\eta(s),s)$. By the condition $(4)$, we can always find such a based loop. Then we choose a path $\gamma:[0,1] \to U \times \mathbb{D}  \setminus \Gamma(\eta)$ where $\gamma(0)=(p,s)$ and $\gamma(1)=(q,o)$ for some $q$. We also choose a loop $\Omega_o$ starting at $q$ around $(0,o)$. Then we have a  homotopy $\Omega_s \sim \gamma^{-1}\Omega_o\gamma$. In other words,  $[\Omega_s]=[\gamma^{-1}\Omega_o\gamma] \in \pi_1(U\times \mathbb{D} \setminus  \Gamma(\eta), p)$. This implies that both $[\Omega_s]$ and $[\Omega_o]$ induce conjugate monodromy representation, hence the deformation is isomonodromic over $U$.
\end{proof}

\begin{lemma}
	\label{lemma:2-5}
	Let $X$ be a general variety from the family $X_{2-5}$, i.\,e. blow-up of a cubic threefold
	in a plane cubic. Then Theorem \ref{theorem:any rank rationality} holds for $X$.
\end{lemma}

\begin{proof}
	We first describe a parameterized Landau--Ginzburg model of $X$. Note that $X$ can be described as a hypersurface
	in a certain toric variety, see~\cite[H.22]{CCGK16}.
    The standard birational change of variables (cf. Example~\ref{eg:2-12}) leads to the pa\-ra\-me\-te\-rized toric Landau--Ginzburg model
	$$
	h=x'+\frac{b'(y'+z'+1)^3(x'+a')}{x'y'z'}
	$$
in $(\CC^*)^3$ with coordinates $x',y',z'$ and parameters $a',b'\neq 0$.
	After the change of variables and parameters
	$$
	x'=bx,\ \ y'=y,\ \ z'=z,\ \ a'=ab,\ \ b'=b,
	$$
	one gets
	\begin{equation*}\label{e:2-5 parLG}
		f_a=h/b=x+\frac{(y+z+1)^3(x+a)}{xyz}.
	\end{equation*}
		By following the compactification procedure in \cite{ChP18}, regarding $a\neq 0$ as a constant, one can construct a family of compactified Landau--Ginzburg models. Consider a pencil of quartic surfaces in $\mathbb{P}^3$
	\[
	\mathcal{S}_a=\{x^2yz+(x+at)(y+z+t)^3=\lambda xyzt\}.\]
	For $\lambda \neq \infty$, we have
	\[
	\begin{aligned}
		& H_{\{x\}}\cdot S_{a, \lambda}=3L_{\{x\}, \{y,z,t\}}+L_{\{x\}, \{at\}} , \\
		& H_{\{y\}}\cdot S_{a, \lambda}=3L_{\{y\}, \{z,t\}}+L_{\{y\}, \{x,at\}},  \\
		& H_{\{z\}}\cdot S_{a, \lambda}=3L_{\{z\}, \{y,t\}}+L_{\{z\}, \{x,at\}}, \\
		& H_{\{t\}}\cdot S_{a, \lambda}=L_{\{x\}, \{t\}}+C, \\
	\end{aligned}
	\]
	where $C$ is the cubic given by $\{t=xyz+(y+z)^3=0\}$. These are all base curves of the pencil $\mathcal{S}_a$.
	
	For $\lambda \neq 0, -1, \infty$, the surface ${S}_{a,\lambda}$ has $6$ isolated singularities
	\[
	\begin{aligned}
		&P_1=[0:1:-1:0], \quad P_2=[0:0:1:-1], \quad P_3=[0:1:0:-1], \\
		&P_4=[1:0:0:0], \quad P_5=[\lambda:1:0:-1], \quad P_6=[\lambda:0:1:-1],
	\end{aligned}
	\]
   so that $S_\lambda$ is irreducible. It is enough to check the case $\lambda=0$. The singular locus of $S_{a,0}$ consists of the line $L_{\{x\}, \{y,z,t\}}$. which contains $P_1, P_2, P_3, P_4$, together with $P_5$ and $P_6$ that collapse to $P_3$ and $P_2$, respectively.  Note that this singular locus is independent on $a$, so that in the compactification procedure introduced in Section \ref{s:CYcptn}, the blown-up loci are flat over $\mathbb{C}^*_a$. Therefore, we obtain a family of compactified Landau--Ginzburg model $\widetilde{Y} \to \mathbb{C} \times \mathbb{D}_a$ such that the locus of the reducible fiber is given by $\{0\}\times \mathbb{D}_a$. By Lemma~\ref{l:isomonodromy}, this local family is isomonodromic.
	
	Finally, we show that the monodromy around the central fiber is strictly quasi--unipotent. Let $f=f_1$ be the anticanonical toric Landau--Ginzburg model for $X$. One can see that the period  $I_f$ is equal to periods of Givental's Landau--Ginzburg model for $X$,
	which, by~\cite[Theorem 0.1]{Gi97} (see also discussion after Corollary 0.4 in~\cite{Gi97})
	is equal to $\widetilde{I}_0^{X,0}$. Thus by writing down its Picard--Fuchs equation by~\cite[Theorem 1.2]{Prz07b},
	the similar argument in Lemma \ref{lemma:1-13} implies that the monodromy around the fiber $S_0$ is non-unipotent.
\end{proof}

\begin{lemma}
	\label{lemma:2-11}
	Let $X$ be a general variety from the family $X_{2-11}$, i.\,e. blow-up of a cubic threefold
	in a line. Then Theorem \ref{theorem:any rank rationality} holds for $X$.
\end{lemma}

\begin{proof}
	The proof is analogous to that of Lemma~\ref{lemma:2-5}. We first describe a parameterized Landau--Ginzburg model of $X$.
	Note that the variety $X$ can be described as a hypersurface
	in a certain toric variety, see~\cite[H.28]{CCGK16}.
    The standard birational change of variables, shifting and scaling the coordinate on the base for Givental's Landau--Ginzburg
model (cf. Example~\ref{eg:2-12}) leads to the parameterized toric Landau--Ginzburg model,
	\begin{equation*}\label{e:2-11 parLG}
	f_a=\frac{a(y+z+1)^3}{xyz}+x+\frac{(y+z+1)^2}{z}.
	\end{equation*}

By following the compactification procedure in \cite{ChP18}, regarding $a\neq 0$ as a constant, one can construct a family of compactified Landau--Ginzburg models. Consider a pencil of quartic surfaces
\[
\mathcal{S}_{a}=\{at(y+z+t)^3+x^2yz+(y+z+t)^2xy=\lambda xyzt\}.\]
For $\lambda \neq 0$, we have
\[
\begin{aligned}
	& H_{\{x\}}\cdot S_{a, \lambda}=3L_{\{x\}, \{y,z,t\}}+L_{\{x\}, \{at\}},  \\
	& H_{\{y\}}\cdot S_{a, \lambda}=3L_{\{y\}, \{z,t\}}+L_{\{y\}, \{at\}},  \\
	& H_{\{z\}}\cdot S_{a, \lambda}=2L_{\{z\}, \{y,t\}}+C_{a,1}, \\
	& H_{\{t\}}\cdot S_{a, \lambda}=L_{\{x\}, \{t\}}+L_{\{y\}, \{t\}}+C_{2}, \\
\end{aligned}
\]
where $C_{a,1}$ is the conic given by $\{z=ayt+xy+at^2=0\}$ and $C_{2}$ is the conic given by $\{t=xz+y^2+2yz+z^2=0\}$. These are all base curves of the pencil $\mathcal{S}_a$.
For every $0 \neq \lambda \in \mathbb{C}$, the surface $S_{a,\lambda}$ has three isolated singular points, that is
$$P_1=[0:1:-1:0],\ \  P_2=[1:0:0:0],\ \ P_3=[0:0:1:-1], \ \ P_4=[0:1:0:-1]
$$
and two other moving singular points $P_5=[\lambda:0:-1:1]$ and $P_6=[\lambda:-1:0:1]$, so that $S_\lambda$ is irreducible. When  $\lambda=0$, the singular locus of $S_{a,0}$ consists of the line $L_{\{x\}, \{y,z,t\}}$. Note that this singular locus is independent on $a$, so that in the compactification procedure introduced in Section \ref{s:CYcptn}, the blown-up loci are flat over $\mathbb{C}^*_a$. Therefore, we will obtain a family of compactified Landau--Ginzburg model $\widetilde{Y} \to \mathbb{C} \times \mathbb{D}_a$ such that the locus of the reducible fiber is given by $\{0\}\times \mathbb{D}_a$. By Lemma \ref{l:isomonodromy}, this local family is isomonodromic.
	
Finally, we show that the monodromy around the central fiber is strictly quasi-unipotent. Let $f=f_1$ be the anticanonical toric Landau--Ginzburg model for $X$. One can see that the series $I_f$ is equal to periods of Givental's Landau--Ginzburg model for $X$,
	which, by~\cite[Theorem 0.1]{Gi97} (see also discussion after Corollary 0.4 in~\cite{Gi97})
	is equal to $\widetilde{I}_0^{X,0}$. Thus by writing down its Picard--Fuchs equation by~\cite[Theorem 1.2]{Prz07b},
	the similar argument in Lemma \ref{lemma:1-13} implies that the monodromy around the reducible fiber is non-unipotent.
%	one can check (see also~\cite{fanosearch}), that there is a reducible fiber with non-unipotent monodromy.
\end{proof}

\begin{remark}
Lemmas~\ref{lemma:1-13},~\ref{lemma:2-5}, and~\ref{lemma:2-11} show that in addition to fibers with isolated singularities
Landau--Ginzburg models of cubic threefold and its blow-ups in a line and in a plane cubic curve the singular fibers
have open book configurations, and a union of two components for the blow-up in an elliptic curve. This agrees
with Mirror Symmetry expectations of changes of Landau--Ginzburg models under blow-ups of Fano varieties.
\end{remark}

	\section{The table}
\label{section:table}
		We comment on the notation that we use in the following table. In the first column we list families of Fano varieties. For a
given family, in the second column we place anticanonical degree $-K^3_X$ of its member, in the third --- its Hodge number $h^{1,2}(X)$.
In the fourth column we place rationality data of the corresponding Fano variety (one can get it from~\cite[\S\S 12.2-12.6]{IP} and references therein); it is ``$+$'', ``$-$'', or ``gen.'' if all varieties of the family rational, non-rational, or rationality of only a general member of the family is known, correspondingly. In particular, in the families $1.2$ and $1.5$ we only deal with the first description of the corresponding Fano variety given in the column ``Brief description''. Finally, in the last column we refer to the number of a particular lemma where the main result of the paper is proven.
	\label{sec-the-table}
	\small{
		\begin{longtable}{|c|c|c|p{9cm}|c|c|}
			\caption{Rationality of smooth Fano threefolds}\label{table:Fanos}\\
			\hline Family & $-K^3$ &$h^{1,2}$&  Brief description & Rat. & Proof \\
			\hline $1.1$ & $2$ & $52$ & a hypersurface in $\mathbb{P}(1,1,1,1,3)$ of
			degree $6$ & $-$ &  \ref{lem:multiple comp} \\
			\hline $1.2$ & $4$&$30$ & a hypersurface in $\mathbb{P}^4$ of degree
			$4$ or\hfill\break a double cover of smooth quadric in $\mathbb{P}^{4}$ branched over a surface of degree $8$ & $-$ & \ref{lem:multiple comp} \\
			\hline $1.3$ & $6$ &$20$& a complete intersection of a quadric and a cubic in
			$\mathbb{P}^{5}$ & $-$ &  \ref{lem:multiple comp}\\
			\hline $1.4$ & $8$ &$14$& a complete intersection of three quadrics
			$\mathbb{P}^{6}$ & $-$ & \ref{lem:multiple comp}\\
			\hline $1.5$ & $10$ &$10$& a section of
			$\mathrm{Gr}(2,5)\subset\mathbb{P}^9$ by
			quadric and linear subspace of dimension~$7$ or \hfill\break
			a double cover of 1-15 with branch locus an anti-canonical divisor
			& $\mathrm{gen.}-$ &  \ref{lem:multiple comp}  \\
			\hline $1.6$ & $12$ &$7$& a section of the Hermitian symmetric space
			$
G/P\subset\mathbb{P}^{15}$ %\hfill\break
of type DIII  by a
			linear
			subspace of dimension~$8$ & $+$ &  \ref{lem:s.n.c case} \\
			\hline $1.7$ & $14$ &$5$ & a section of
			$\mathrm{Gr}(2,6)\subset\mathbb{P}^{14}$
			by a linear subspace of codimension~$5$ & $-$ &  \ref{lemma:1-7}\\
			\hline $1.8$ & $16$ & $3$& a section of the Hermitian symmetric space
			$
G/P\subset \mathbb{P}^{19}$ %\hfill\break
of type CI  by a linear subspace of dimension~$10$ & $+$ &    \ref{lem:s.n.c case} \\
			\hline $1.9$ & $18$ & $2$& a section of the $5$-dimensional rational
			homogeneous contact %\hfill\break
manifold
			$G_2/P\subset\mathbb{P}^{13}$  by
			a linear subspace of dimension~$11$ & $+$ &   \ref{lem:s.n.c case} \\
			\hline $1.10$ & $22$ &$0$& a zero locus of three sections of the rank
			$3$ vector bundle $\bigwedge^2\mathcal{Q}$, %\hfill\break
where
			$\mathcal{Q}$ is
			the universal quotient bundle on $\mathrm{Gr}(7,3)$ & $+$ &
			 \ref{lem: no component}\\
			\hline $1.11$ & $8$ &$21$& $V_{1}$ that is a hypersurface in
			$\mathbb{P}(1,1,1,2,3)$ of degree $6$ & $-$ &
			 \ref{lem:multiple comp} \\
			\hline $1.12$ & $16$ &$10$& $V_{2}$ that is a hypersurface in
			$\mathbb{P}(1,1,1,1,2)$ of degree $4$ & $-$ &  \ref{lem:multiple comp} \\
			\hline $1.13$ & $24$ &$5$& $V_{3}$ that is a hypersurface in
			$\mathbb{P}^{4}$ of degree $3$ & $-$ &  \ref{lemma:1-13} \\
			\hline $1.14$ & $32$ &$2$& $V_{4}$ that is a complete intersection of two
			quadrics in $\mathbb{P}^{5}$ & $+$ &   \ref{lem:s.n.c case}
			\\
			\hline $1.15$ & $40$ &$0$& $V_{5}$ that is a section of
			$\mathrm{Gr}(2,5)\subset\mathbb{P}^9$ by linear subspace of
			codimension $3$ & $+$ &  \ref{lem: no component}\\
			\hline $1.16$ & $54$ &$0$& $Q$ that is a hypersurface in $\mathbb{P}^{4}$
			of degree $2$ & $+$ &  \ref{lem: no component} \\
			\hline $1.17$ & $64$ &$0$& $\mathbb{P}^{3}$ & $+$ &  \ref{lem: no component}\\
			\hline $2.1$ & $4$ &$22$& a blow-up of the Fano threefold $V_1$ in an
			elliptic curve %\hfill\break
that is an intersection of two divisors
			from $|-\frac{1}{2}K_{V_1}|$ &$-$ &  \ref{lem:multiple comp} \\
			\hline $2.2$ & $6$ &$20$& a double cover of
			$\mathbb{P}^1\times\mathbb{P}^2$
			whose branch locus is a divisor of bidegree $(2, 4)$ & $-$ &  \ref{lem:multiple comp}\\
			\hline $2.3$ & $8$ &$11$& the blow-up of the Fano threefold $V_2$ in an
			elliptic curve %\hfill\break
that is an intersection of two divisors
			from $|-\frac{1}{2}K_{V_2}|$ & $-$ &  \ref{lem:multiple comp}\\
			\hline $2.4$ & $10$ &$10$& the blow-up of $\mathbb{P}^3$ along an
			intersection of two cubics & $+$ &   \ref{lem:s.n.c case}
			\\
			\hline $2.5$ & $12$ &$6$& the blow-up of the threefold
			$V_3\subset\mathbb{P}^4$ along a plane cubic & $-$ &  \ref{lemma:2-5} \\
			\hline $2.6$ & $12$ &$9$& a divisor on
			$\mathbb{P}^2\times\mathbb{P}^2$ of bidegree $(2, 2)$
			or\hfill\break a double cover of $W$ whose branch locus
			is a surface in $|-K_W|$  & $-$ &  \ref{lem:multiple comp}\\
			\hline $2.7$ & $14$ &$5$& the blow-up of $Q$ along the intersection of
			two divisors from $|\mathcal{O}_Q (2)|$ & $+$ &   \ref{lem:s.n.c case}\\
			\hline $2.8$ & $14$ &$9$& a double cover of $V_7$ whose branch locus
			is a surface
			in $|-K_{V_7}|$ & $-$ &   \ref{lem:multiple comp}\\
			\hline $2.9$ & $16$ &$5$& the blow-up of $\mathbb{P}^3$ along a curve
			of degree $7$ and genus~$5$\hfill\break which is an intersection
			of cubics & $+$ &   \ref{lem:s.n.c case} \\
			\hline $2.10$ & $16$ &$3$& the blow-up of $V_4\subset\mathbb{P}^5$
			along an elliptic curve\hfill\break which is an intersection of
			two hyperplane
			sections & $+$ &   \ref{lem:s.n.c case} \\
			\hline $2.11$ & $18$ &$5$& the blow-up of $V_3$ along a line & $-$ &  \ref{lemma:2-11}
			\\
			\hline $2.12$ & $20$ &$3$& the blow-up of $\mathbb{P}^3$ along a curve
			of degree $6$ and genus~$3$\hfill\break which is an intersection
			of cubics & $+$ &   \ref{lem:s.n.c case} \\
			\hline $2.13$ & $20$ &$2$& the blow-up of $Q\subset\mathbb{P}^4$ along
			a curve of degree $6$ and genus $2$ & $+$ &   \ref{lem:s.n.c case}\\
			\hline $2.14$ & $20$ &$1$& the blow-up of $V_5\subset\mathbb{P}^6$ along
			an elliptic curve\hfill\break which is an intersection of two
			hyperplane sections & $+$ &   \ref{lem:s.n.c case} \\
			\hline $2.15$ & $22$ &$4$& the blow-up of $\mathbb{P}^3$ along the
			intersection of a quadric and a cubic surfaces & $+$ &   \ref{lem:s.n.c case} \\
			\hline $2.16$ & $22$ &$2$& the blow-up of $V_4\subset\mathbb{P}^5$
			along a conic & $+$ &  \ref{lem:s.n.c case} \\
			\hline $2.17$ & $24$ &$1$& the blow-up of $Q\subset\mathbb{P}^4$ along
			an elliptic curve of degree~$5$ & $+$ &   \ref{lem:s.n.c case} \\
			\hline $2.18$ & $24$ &$2$& a double cover of
			$\mathbb{P}^1\times\mathbb{P}^2$ whose branch locus is a divisor of
			bidegree $(2, 2)$ & $+$ &   \ref{lem:s.n.c case} \\
			\hline $2.19$ & $26$ &$2$& the blow-up of $V_4\subset\mathbb{P}^5$ along a
			line & $+$ &   \ref{lem:s.n.c case} \\
			\hline $2.20$ & $26$ &$0$& the blow-up of $V_5\subset\mathbb{P}^6$
			along a twisted cubic & $+$ &  \ref{lem: no component}\\
			\hline $2.21$ & $28$ &$0$& the blowup of $Q\subset\mathbb{P}^4$ along
			a twisted quartic & $+$ &   \ref{lem: no component}\\
			\hline $2.22$ & $30$ &$0$& the blow-up of $V_5\subset\mathbb{P}^6$
			along a conic & $+$ &  \ref{lem: no component} \\
			\hline $2.23$ & $30$ &$1$& the blow-up of $Q\subset\mathbb{P}^4$ along a
			curve of degree $4$ that is an intersection of a surface in
			$|\mathcal{O}_{\mathbb{P}^{4}}(1)\vert_{Q}|$ and a surface in
			$|\mathcal{O}_{\mathbb{P}^{4}}(2)\vert_{Q}|$ & $+$ &   \ref{lem:s.n.c case} \\
			\hline $2.24$ & $30$ &$0$& a divisor on $\mathbb{P}^2\times\mathbb{P}^2$
			of bidegree $(1, 2)$ & $+$ &  \ref{lem: no component} \\
			\hline $2.25$ & $32$ &$1$& the blow-up of $\mathbb{P}^3$ along an elliptic
			curve which is an intersection of two quadrics & $+$ &   \ref{lem:s.n.c case} \\
			\hline $2.26$ & $34$ &$0$& the blow-up of the threefold
			$V_5\subset\mathbb{P}^6$ along a line & $+$ & \ref{lem: no component} \\
			\hline $2.27$ & $38$ &$0$& the blow-up of $\mathbb{P}^3$ along a twisted
			cubic  & $+$ &  \ref{lem: no component} \\
			\hline $2.28$ & $40$ &$1$& the blow-up of $\mathbb{P}^3$ along a plane
			cubic & $+$ &   \ref{lem:s.n.c case} \\
			\hline $2.29$ & $40$ &$0$& the blow-up of $Q\subset\mathbb{P}^4$ along a
			conic & $+$ &  \ref{lem: no component}\\
			\hline $2.30$ & $46$ &$0$& the blow-up of $\mathbb{P}^3$ along a conic   &
			$+$ & \ref{lem: no component} \\
			\hline $2.31$ & $46$ &$0$& the blow-up of $Q\subset\mathbb{P}^4$ along a
			line & $+$ & \ref{lem: no component} \\
			\hline $2.32$ & $48$ &$0$& $W$ that is a divisor on
			$\mathbb{P}^2\times\mathbb{P}^2$ of bidegree $(1, 1)$  & $+$ &
			 \ref{lem: no component}\\
			\hline $2.33$ & $54$ &$0$& the blow-up of $\mathbb{P}^3$ along a line  &
			$+$ &  \ref{lem: no component} \\
			\hline $2.34$ & $54$ &$0$& $\mathbb{P}^1\times\mathbb{P}^2$  & $+$ &
			 \ref{lem: no component}\\
			\hline $2.35$ & $56$ &$0$&
			$V_7\cong\mathbb{P}(\mathcal{O}_{\mathbb{P}^2}\oplus\mathcal{O}_{\mathbb{P}^2}(1))$
			& $+$ &  \ref{lem: no component} \\
			\hline $2.36$ & $62$ &$0$&
			$\mathbb{P}(\mathcal{O}_{\mathbb{P}^2}\oplus\mathcal{O}_{\mathbb{P}^2}(2))$
			& $+$ &  \ref{lem: no component}\\
			\hline $3.1$ & $12$ &$8$& a double cover of
			$\mathbb{P}^1\times\mathbb{P}^1\times\mathbb{P}^1$ branched in a
			divisor of type $(2, 2, 2)$ & $-$ & \ref{lem:multiple comp} \\
			\hline $3.2$ & $14$ &$3$& a divisor on a $\mathbb{P}^{2}$-bundle
			$\mathbb{P}(\mathcal{O}_{\mathbb{P}^1\times\mathbb{P}^1}\oplus\mathcal{O}_{\mathbb{P}^1\times\mathbb{P}^1}(-1,-1)\oplus\mathcal{O}_{\mathbb{P}^1\times\mathbb{P}^1}(-1,-1))$
			such that $X\in|L^{{}\otimes
				2}\otimes\mathcal{O}_{\mathbb{P}^{1}\times\mathbb{P}^{1}}(2,3)|$,
			where $L$ is the tautological line bundle & $+$ &   \ref{lem:3-2}  \\
			\hline $3.3$ & $18$ &$3$& a divisor on
			$\mathbb{P}^1\times\mathbb{P}^1\times\mathbb{P}^2$ of type $(1,
			1, 2)$ & $+$ &   \ref{lem:s.n.c case} \\
			\hline $3.4$ & $18$ &$2$& the blow-up of the Fano threefold $Y$ from the
			family $\textnumero\ 2.18$ along a smooth fiber of the
			composition $Y\to\mathbb{P}^1\times\mathbb{P}^2\to\mathbb{P}^2$ of the
			double cover with the projection & $+$ &   \ref{lem:s.n.c case}
			\\
			\hline $3.5$ & $20$ &$0$& the blow-up of $\mathbb{P}^1\times\mathbb{P}^2$
			along a curve $C$ of bidegree $(5, 2)$ %\hfill\break
such that the
			composition  $C\hookrightarrow\mathbb{P}^1\times\mathbb{P}^2\to\mathbb{P}^2$
			is an embedding & $+$ &  \ref{lem: no component} \\
			\hline $3.6$ & $22$ &$1$& the blow-up of $\mathbb{P}^3$ along a disjoint
			union of a line and an elliptic curve of degree~$4$ & $+$ &   \ref{lem:s.n.c case}\\
			\hline $3.7$ & $24$ &$1$& the blow-up of the threefold $W$ along an
			elliptic curve\hfill\break that is an intersection of two  divisors
			from $|-\frac{1}{2}K_W|$  & $+$ &   \ref{lem:s.n.c case}
			\\
			\hline $3.8$ & $24$ &$0$& a divisor in
			$|(\alpha\circ\pi_1)^*(\mathcal{O}_{\mathbb{P}^2}(1))\otimes\pi_2^*(\mathcal{O}_{\mathbb{P}^2}(2))|$,
			where $\pi_{1}\colon\mathbb{F}_1\times\mathbb{P}^2\to\mathbb{F}_1$
			and $\pi_{2}\colon\mathbb{F}_1\times\mathbb{P}^2\to\mathbb{P}^2$ are
			projections, and $\alpha\colon\mathbb{F}_1\to\mathbb{P}^2$ is a blow-up of a point & $+$ & \ref{lem: no component} \\
			\hline $3.9$ & $26$ &$3$& the blow-up of a cone $W_4\subset\mathbb{P}^6$
			over the Veronese surface  $R_4\subset\mathbb{P}^5$ %\hfill\break
with
			center in a disjoint union of the vertex and a quartic on
			$R_4\cong\mathbb{P}^2$ & $+$ &   \ref{lem:s.n.c case} \\
			\hline $3.10$ & $26$ &$0$& the blow-up of $Q\subset\mathbb{P}^4$ along a
			disjoint union of two conics & $+$ &  \ref{lem: no component}\\
			\hline $3.11$ & $28$ &$1$& the blow-up of the threefold $V_7$ along an
			elliptic curve\hfill\break that is an intersection of  two divisors
			from $|-\frac{1}{2}K_{V_7}|$ & $+$ &   \ref{lem:s.n.c case} \\
			\hline $3.12$ & $28$ &$0$& the blow-up of $\mathbb{P}^3$ along a disjoint
			union of a line and a twisted cubic & $+$ &  \ref{lem: no component} \\
			\hline $3.13$ & $30$ &$0$& the blow-up of
			$W\subset\mathbb{P}^2\times\mathbb{P}^2$ along a curve $C$ of
			bidegree $(2, 2)$ %\hfill\break
such that
			$\pi_{1}(C)\subset\mathbb{P}^2$ and
			$\pi_{2}(C)\subset\mathbb{P}^{2}$ are irreducible
			conics, %\hfill\break
where $\pi_{1}\colon W\to\mathbb{P}^2$ and
			$\pi_{2}\colon W\to\mathbb{P}^2$ are
			natural projections & $+$ &  \ref{lem: no component} \\
			\hline $3.14$ & $32$ &$1$& the blow-up of $\mathbb{P}^3$ along a
			disjoint union of a plane cubic curve that is
			contained in a plane
			$\Pi\subset\mathbb{P}^{3}$ and a point that is not contained in $\Pi$
			& $+$ &   \ref{lem:s.n.c case} \\
			\hline $3.15$ & $32$ &$0$& the blow-up of $Q\subset\mathbb{P}^4$ along a
			disjoint union of a line and a conic & $+$ &  \ref{lem: no component} \\
			\hline $3.16$ & $34$ &$0$& the blow-up of $V_7$ along a proper
			transform via the blow-up $\alpha\colon
			V_7\to\mathbb{P}^3$ of a twisted cubic
			passing through the center of the blow-up $\alpha$ & $+$ &  \ref{lem: no component} \\
			\hline $3.17$ & $36$ &$0$& a divisor on
			$\mathbb{P}^1\times\mathbb{P}^1\times\mathbb{P}^2$ of type $(1,
			1, 1)$ & $+$ &  \ref{lem: no component} \\
			\hline $3.18$ & $36$ &$0$& the blow-up of $\mathbb{P}^3$ along a disjoint
			union of a line and a conic & $+$ &  \ref{lem: no component} \\
			\hline $3.19$ & $38$ &$0$& the blow-up of $Q\subset\mathbb{P}^4$ at two
			non-collinear points & $+$ & \ref{lem: no component}\\
			\hline $3.20$ & $38$ &$0$& the blow-up of $Q\subset\mathbb{P}^4$ along a
			disjoint union of two lines & $+$ & \ref{lem: no component} \\
			\hline $3.21$ & $38$ &$0$& the blow-up of $\mathbb{P}^1\times\mathbb{P}^2$
			along a curve of bidegree  $(2, 1)$ & $+$ &  \ref{lem: no component} \\
			\hline $3.22$ & $40$ &$0$& the blow-up of $\mathbb{P}^1\times\mathbb{P}^2$
			along a conic in a fiber of the projection
			$\mathbb{P}^{1}\times\mathbb{P}^2\to\mathbb{P}^1$ & $+$ & \ref{lem: no component} \\
			\hline $3.23$ & $42$ &$0$& the blow-up of $V_7$ along a proper transform
			via the blow-up $\alpha\colon V_7\to\mathbb{P}^3$ of an
			irreducible conic passing through the center of the blow-up $\alpha$ &
			$+$ &  \ref{lem: no component} \\
			\hline $3.24$ & $42$ &$0$& $W\times_{\mathbb{P}^2}\mathbb{F}_1$, where
			$W\to\mathbb{P}^2$ is a $\mathbb{P}^1$-bundle and
			$\mathbb{F}_1\to\mathbb{P}^2$ is the
			blow-up & $+$ &  \ref{lem: no component}\\
			\hline $3.25$ & $44$ &$0$& the blow-up of $\mathbb{P}^3$ along a disjoint
			union of two lines & $+$ & \ref{lem: no component} \\
			\hline $3.26$ & $46$ &$0$& the blow-up of $\mathbb{P}^3$ with center in a
			disjoint union of a point and a line  & $+$ & \ref{lem: no component}
			\\
			\hline $3.27$ & $48$ &$0$&
			$\mathbb{P}^1\times\mathbb{P}^1\times\mathbb{P}^1$  & $+$ & \ref{lem: no component}
			\\
			\hline $3.28$ & $48$ &$0$& $\mathbb{P}^1\times\mathbb{F}_1$  & $+$ & \ref{lem: no component}
			\\
			\hline $3.29$ & $50$ &$0$& the blow-up of the Fano threefold $V_7$
			along a line in $E\cong\mathbb{P}^2$, %\hfill\break
where $E$ is the
			exceptional divisor of the
			blow-up $V_7\to\mathbb{P}^3$  & $+$ &   \ref{lem: no component}\\
			\hline $3.30$ & $50$ &$0$& the blow-up of $V_7$ along a proper transform
			via the blow-up $\alpha\colon V_7\to\mathbb{P}^3$ of a
			line that passes through the center of the blow-up $\alpha$   & $+$ &
			 \ref{lem: no component} \\
			\hline $3.31$ & $52$ &$0$& the blow-up of a cone over a smooth quadric in
			$\mathbb{P}^3$ at the vertex & $+$ &  \ref{lem: no component} \\
			\hline $4.1$ & $24$ &$1$& divisor on
			$\mathbb{P}^1\times\mathbb{P}^1\times\mathbb{P}^1\times\mathbb{P}^1$
			of multidegree $(1, 1, 1, 1)$ & $+$ &    \ref{lem:s.n.c case}\\
			\hline $4.2$ & $28$ &$1$& the blow-up of the cone over a smooth quadric
			$S\subset\nolinebreak\mathbb{P}^3$\hfill\break along a disjoint union
			of the vertex and an elliptic curve on $S$  & $+$ &   \ref{lem:s.n.c case} \\
			\hline $4.3$ & $30$ &$0$& the blow-up of
			$\mathbb{P}^1\times\mathbb{P}^1\times\mathbb{P}^1$ along a curve of
			type $(1, 1, 2)$  & $+$ &  \ref{lem: no component}\\
			\hline $4.4$ & $32$ &$0$& the blow-up of the smooth Fano threefold $Y$
			from the family $\textnumero\ 3.19$ along the proper
			transform of a conic on the quadric $Q\subset\mathbb{P}^4$ %\hfill\break
			that passes through the both centers of the blow-up $Y\to Q$ & $+$ & \ref{lem: no component}
			\\
			\hline $4.5$ & $32$ &$0$& the blow-up of $\mathbb{P}^1\times\mathbb{P}^2$
			along a disjoint union of\hfill\break two irreducible curves of
			bidegree $(2, 1)$ and $(1, 0)$ & $+$ &  \ref{lem: no component}\\
			\hline $4.6$ & $34$ &$0$& the blow-up of $\mathbb{P}^3$ along a disjoint
			union of three lines & $+$ &  \ref{lem: no component} \\
			\hline $4.7$ & $36$ &$0$& the blow-up of
			$W\subset\mathbb{P}^2\times\mathbb{P}^2$ along a disjoint union
			of\hfill\break two curves of bidegree $(0, 1)$ and $(1, 0)$ &+ & \ref{lem: no component}
			\\
			\hline $4.8$ & $38$ &$0$& the blow-up of
			$\mathbb{P}^1\times\mathbb{P}^1\times\mathbb{P}^1$ along a curve of
			type $(0, 1, 1)$  & $+$ & \ref{lem: no component}  \\
			\hline $4.9$ & $40$ &$0$& the blow-up of the smooth Fano threefold $Y$
			from the family
			$\textnumero\ 3.25$ along a curve that is contracted
			by the blow-up $Y \to\mathbb{P}^3$  &  $+$ &  \ref{lem: no component}\\
			\hline $4.10$ & $42$ &$0$& $\mathbb{P}^1\times S_7$  & $+$ & \ref{lem: no component}
			\\
			\hline $4.11$ & $44$ &$0$& the blow-up of
			$\mathbb{P}^1\times\mathbb{F}_1$ along a curve $C\cong\mathbb{P}^1$ such
			that $C$ is contained in a fiber $F\cong\mathbb{F}_{1}$ of
			the projection
			$\mathbb{P}^1\times\mathbb{F}_{1}\to\mathbb{P}^1$ and $C\cdot C=-1$ on
			$F$ & $+$ &  \ref{lem: no component}\\
			\hline $4.12$ & $46$ &$0$& the blow-up of the smooth Fano threefold
			$Y$ from the family $\textnumero\ 2.33$ along two curves that are
			contracted by the blow-up
			$Y\to\mathbb{P}^3$ & $+$ &   \ref{lem: no component}\\
			\hline $4.13$ & $26$ &$0$& the blow-up of
			$\mathbb{P}^1\times\mathbb{P}^1\times\mathbb{P}^1$ along a curve of
			type $(1, 1, 3)$ & $+$ & \ref{lem: no component} \\
			\hline $5.1$ & $28$ &$0$& the blow-up of the smooth Fano threefold $Y$
			from the family $\textnumero\ 2.29$ along three curves that are
			contracted by the
			blow-up $Y\to Q$ & $+$ & \ref{lem: no component} \\
			\hline $5.2$ & $36$ &$0$& the blow-up of the smooth Fano threefold $Y$
			in the family $3.25$ along two curves $C_{1}\ne
			C_{2}$ that are contracted by the blow-up $\phi\colon
			Y\to\nolinebreak\mathbb{P}^3$ and that are contained in the
			same exceptional divisor of the blow-up $\phi$  &
			$+$ & \ref{lem: no component} \\
			\hline $5.3$ & $36$ &$0$& $\mathbb{P}^1\times S_6$  & $+$ & \ref{lem: no component}
			\\
			\hline $6.1$ & $30$ &$0$& $\mathbb{P}^1\times S_5$  & $+$ & \ref{lem: no component}
			\\
			\hline $7.1$ & $24$ &$0$& $\mathbb{P}^1\times S_4$  & $+$ & \ref{lem: no component}
			\\
			\hline $8.1$ & $18$ &$0$& $\mathbb{P}^1\times S_3$  & $+$ & \ref{lem: no component}
			\\
			\hline $9.1$ & $12$ &$0$& $\mathbb{P}^1\times S_2$  & $+$ & \ref{lem: no component}
			\\
			\hline $10.1$ & $6$ &$0$& $\mathbb{P}^1\times S_1$  &  $+$ & \ref{lem: no component}
			\\
			\hline
		\end{longtable}
	}

\end{document}